\documentclass[10pt,a4paper]{amsart}
\usepackage[utf8]{inputenc}
\usepackage[english]{babel}
\usepackage{amsmath}
\usepackage{amsfonts}
\usepackage{amssymb}
\usepackage{bbm}
\usepackage{mathrsfs}
\usepackage[all]{xy}
\usepackage{bm}
\usepackage[left=2cm,right=2cm,top=2cm,bottom=2cm]{geometry}
\usepackage{ytableau}
\author{Geoffrey Powell}
\title{On the $\finj$-homology of the injective cogenerators}
\date{}
\thanks{This work was partially supported by the ANR Project {\em ChroK}, {\tt ANR-16-CE40-0003}.}
\subjclass[2000]{}
\newtheorem{THM}{Theorem}
\newtheorem{PROP}[THM]{Proposition}
\newtheorem{thm}{Theorem}[section]
\newtheorem{prop}[thm]{Proposition}
\newtheorem{cor}[thm]{Corollary}
\newtheorem{lem}[thm]{Lemma}
\theoremstyle{definition}
\newtheorem{defn}[thm]{Definition}
\newtheorem{exam}[thm]{Example}
\theoremstyle{remark}
\newtheorem{rem}[thm]{Remark}
\newtheorem{nota}[thm]{Notation}
\newtheorem{hyp}[thm]{Hypothesis}
\newtheorem{conj}[thm]{Conjecture}
\newcommand{\f}{\mathcal{F}}
\newcommand{\hs}[1]{\widehat{#1}}
\newcommand{\vstrip}[1]{\widetilde{#1}}
\newcommand{\crit}{\mathfrak{Crit}}
\newcommand{\M}{\mathfrak{M}}
\newcommand{\pair}{\mathfrak{P}}
\newcommand{\fbmod}{\mathcal{F}(\fb)}
\newcommand{\fimod}{\mathcal{F}(\finj)}
\newcommand{\conv}{\odot}
\newcommand{\kz}{\mathfrak{Kz}^\finj}
\newcommand{\ori}{\mathrm{Or}}
\newcommand{\hfi}{\hom_\finj}
\newcommand{\aut}{\mathrm{Aut}}
\newcommand{\fsets}{\mathrm{Set}^{\mathrm{f}}}
\newcommand{\vs}{\mathcal{V}_\kring}
\newcommand{\fvs}{\vs^{\mathrm{f}}}
\newcommand{\trans}{^{\mathrm{tr}}}
\renewcommand{\hom}{\mathrm{Hom}}
\newcommand{\sym}{\mathfrak{S}}
\newcommand{\nat}{\mathbb{N}}
\newcommand{\zed}{\mathbb{Z}}
\newcommand{\op}{^\mathrm{op}}
\newcommand{\ob}{\mathrm{Ob}\hspace{2pt}}
\newcommand{\kring}{\mathbbm{k}}
\newcommand{\dash}{\hspace{-2pt}-\hspace{-2pt}}
\newcommand{\modules}{\mathrm{mod}}
\newcommand{\fb}{{\bm{\Sigma}}}
\newcommand{\tr}{\mathrm{Tr}}
\newcommand{\tab}{\mathrm{Tab}}
\newcommand{\rev}{^{\mathrm{rev}}}
\newcommand{\tspec}{\mathbb{T}}
\newcommand{\coeff}{\theta}
\newcommand{\finj}{\mathbf{FI}}
\newcommand{\re}{r^{\epsilon}}
\numberwithin{equation}{section}
\begin{document}
\footnotetext{https://orcid.org/0000-0003-2564-1202}

\begin{abstract}
The purpose of this paper is to give information on the $\finj$-homology of the standard injective cogenerators of the category of $\finj$-modules, where $\finj$ is the category of finite sets and injections.

Working over a field $\kring$ of characteristic zero, a full calculation is given in homological degree zero and a conjectural description in higher homological degree.

The proof of the main theorem reduces to a calculation in representation theory of the symmetric groups, exploiting the Young orthonormal basis. 
\end{abstract}

\maketitle

\section{Introduction}

The category $\finj$ of finite sets and injections is a fundamental tool for relating representations of the symmetric groups, which appear as the automorphism groups of the objects of $\finj$. In particular, for any commutative, unital ring, one can consider $\f (\finj)$, the category of functors from $\finj$ to $\kring$-modules, often known as the category of $\finj$-modules. This provides a framework for studying phenomena such as representation stability, for example. In this paper,  $\kring$ is taken to be a field of characteristic zero, since the proof of the main results exploits the representation theory of the symmetric groups in characteristic zero. 

The category $\f (\finj)$ has enough projectives and enough injectives, which are provided by Yoneda's lemma. For $b \in \nat$, writing $\mathbf{b} := \{ 1, \ldots , b\}$, the associated standard projective generator  is $\kring \hfi (\mathbf{b}, -)$, the $\kring$-linearization of the set-valued functor $\hfi (\mathbf{b}, -)$. Correspondingly, one has the standard injective cogenerator given by $\kring ^{\hfi (-, \mathbf{b})}$, formed by taking set maps with values in $\kring$. Contrary to  the standard projectives, these are {\em torsion} $\finj$-modules; in particular, for  $a \in \nat$, and $\mathbf{a}:= \{1, \ldots , a \}$, $\kring ^{\hfi (-, \mathbf{b})}(\mathbf{a})=\kring ^{\hfi (\mathbf{a}, \mathbf{b})}$, which is zero  for $a>b$. 

For the purposes of this paper (with its application in \cite{P_wall} in view), it is convenient to describe the above injective cogenerator as 
\[
\kring \hfi (-,\mathbf{b})\trans : \finj \rightarrow \kring \dash \modules,
\]
where $\trans$ indicatesthe following `transpose' {\em covariant} structure:  $\kring \hfi (-, \mathbf{b})\trans(\mathbf{a})$ is the $\kring$-vector space with basis $\{ [f] \ | \  f \in \hfi (\mathbf{a}, \mathbf{b}) \}$. Then for $i : \mathbf{a} \hookrightarrow \mathbf{a'}$ a map in $\finj$:
\[
\kring \hfi (i, \mathbf{b}) [f] = \sum _{\substack{f' \in \hfi (\mathbf{a'}, \mathbf{b}) \\
f' \circ i = f }}[f'].
\]
There is more structure, namely one can take into account the action of $\sym_b:= \aut (\mathbf{b})$, so that the above is a functor from $\finj$ to $\kring\sym_b$-modules.  
 
The category $\fb$ of finite sets and bijections identifies  as the maximal subgroupoid of $\finj$. Functors from $\fb$ to $\kring$-vector spaces are denoted by $\f (\fb)$ and there is an extension by zero functor $\f (\fb) \rightarrow \f (\finj)$. This has left adjoint $H_0 ^\finj : \f (\finj) \rightarrow \f (\fb)$ and, by definition, $\finj$-homology 
$H_ * ^\finj$ is given by forming the left derived functors.   

Behaviour of $\finj$-homology on the standard projective is easily understood:
\[
\big( H^\finj_0 \kring \hfi ( \mathbf{b},-) \big) (\mathbf{a}) = 
\left\{
\begin{array}{ll}
\kring \sym_b & a=b \\
0 & \mbox{otherwise}.
\end{array}
\right.  
\]
Clearly higher $\finj$-homology of $\kring \hfi ( \mathbf{b},-)$ vanishes, since these are defined as left derived functors.

{\em A contrario}, the case of the standard injective cogenerators $ \kring \hfi (-,\mathbf{b})\trans $, for $b\in \nat$, is much more complicated.  In particular, $H_0^\finj  \kring \hfi (-,\mathbf{b})\trans $ is not in general supported on $\mathbf{b}$.

The main result of this paper, Theorem \ref{thm:coker}, calculates this $H_0^\finj$; in the statement, for a partition $\lambda$, $S^\lambda$ is the simple $\sym_{|\lambda|}$-module indexed by $\lambda$ (see Section \ref{sect:rep} for some recollections on the representation theory of the symmetric groups and references).

\begin{THM}
\label{THM1}
For $1 \leq a\leq b$, there is an isomorphism of $\sym_b\times \sym_a$-modules:
\[
 (H^\finj_0 \kring \hfi (-, \mathbf{b})\trans ) (\mathbf{a})
\cong 
\bigoplus_{\substack{\lambda \vdash b \\ \lambda_1 =b-a}}
S^\lambda \boxtimes S^{\hs{\lambda}},
\]
where $\hs{\lambda}$ is the partition obtained from $\lambda$ by removing $\lambda_1$.

In particular, this is multiplicity-free.
\end{THM}

This result has a striking aspect that can be paraphrased as stating that the above $\finj$-homology in homological degree zero is as small as it can be, given the structure of the representations involved (cf. Lemma \ref{lem:comp_factor_max_hor_strip}). That such a result should be true was suggested by calculations carried out by the author using Sage \cite{sagemath} in January 2022.

The starting point is the identification as representations 
\[
\kring \hfi (\mathbf{a}, \mathbf{b})\trans 
\cong 
\bigoplus_{\substack{\lambda \vdash b, \nu \vdash a \\ \hs{\lambda}\preceq \nu \preceq \lambda}}
S^\lambda \boxtimes S^{\nu}.
\]
Here the condition $\hs{\lambda}\preceq \nu \preceq \lambda$ is equivalent to the skew partition $\lambda/ \nu$ being a `horizontal strip' (i.e., having skew Young diagram with at most one box in each column). 

The proof of Theorem \ref{THM1} boils down to a property of the skew representation $S^{\lambda/\nu}$ when $\nu \neq \hs{\lambda}$, as is explained in Section \ref{subsect:first_reduct}. Since $\lambda /\nu$ is a horizontal strip, $S^{\lambda/\nu}$ is a permutation representation; the difficulty in proving the Theorem stems from the fact that the explicit element which intervenes is defined in terms of the Young orthonormal basis, which is {\em not} a permutation basis.  

This proof takes up most of Section \ref{sect:calc} and exploits elementary, brute force techniques. The author believes that there should be a better proof of this result and suspects that the result should already occur in some form in the literature.

Emboldened by Theorem \ref{THM1}, it is natural to consider the $\finj$-homology in higher homological degree. Proposition \ref{prop:subset_H_finj} gives the following lower bound for this (the notation used in the statement is introduced in Section \ref{sect:homology}):

\begin{PROP}
For $b \geq  a\in \nat$ and homological degree $n$, there is an inclusion of $\sym_b \times \sym_a $-modules:
\[
\bigoplus_{\substack{(\lambda, \mu) \in \crit (a,b) \\
|\mu/(\lambda \cap \mu)|=n }}
S^\lambda \boxtimes S^\mu 
\subset 
H_n ^\finj (\kring \hfi(-, \mathbf{b}) ) (\mathbf{a}).
\]
For $n=0$ this is an isomorphism.
\end{PROP}

An optimistic conjecture is that this Proposition actually gives the full $\finj$-homology. Since this is not the main thrust of the paper, the proposed strategy for establishing the conjecture is only outlined here.

The appendix, Section \ref{sect:kz_schur} revisits these structures using the Schur-Weyl correspondence, i.e., by passing to Schur (bi)functors. This allows the Koszul complex that calculates the $\finj$-homology to be made completely explicit (see Theorem \ref{thm:schur_koszul}) and also introduces further  structure (see Corollary \ref{cor:module}).  This is significant in the application in \cite{P_wall}, where it is encoded in the wall categories appearing there.

\subsection{Acknowledgement}
The author is grateful to Christine Vespa for comments which lead to the natural description of the main players.

\section{Ingredients}

The category of finite sets and injective maps is denoted $\finj$ and its maximal subgroupoid by $\fb$, which identifies with the category of finite sets and bijections. 

\begin{nota}
\ 
\begin{enumerate}
\item 
For $t \in \nat$, $\mathbf{t}$ denotes the finite set $\{ 1, \ldots, t \}$, which can be considered as an object of $\fb$ (and hence of $\finj$). By convention, $\mathbf{0}= \emptyset$.
\item 
For $s \leq t$, $\iota_{s,t} \in \hfi (\mathbf{s}, \mathbf{t})$ denotes the canonical inclusion
 $\{ 1, \ldots, s \} \subset \{ 1, \ldots, t \}$.
\end{enumerate}
\end{nota}

Morphisms in $\finj$ yield a functor 
\[
\hfi (-, -) \ : \ \finj\op \times \finj \rightarrow \fsets,
\]
where $\fsets$ is the category of finite sets. The $\kring$-linearization gives a functor $\kring \hfi (-,-)$ to  $\fvs \subset \vs$, the full subcategory of finite-dimensional $\kring$-vector spaces.

In addition,  if one restricts the second variable to  $\fb \subset \finj$, $\kring \hfi (-,-)$ has a {\em covariant} functoriality with respect to the first variable:

\begin{nota}
Denote by $\kring \hfi (-,-)\trans$ the functor $\finj \times \fb \rightarrow \fvs$ given by 
\[
\kring \hfi (-,-)\trans : (\mathbf{a}, \mathbf{b}) \mapsto \kring \hfi (\mathbf{a}, \mathbf{b})
\]
where, for  $i : \mathbf{a} \hookrightarrow \mathbf{a'}$ in $\finj$ and  $f \in \hfi (\mathbf{a}, \mathbf{b})$, $\kring \hfi (i , \mathbf{b}) \trans [f]$ is the sum 
$\sum [f']$, where $f' \in \hfi (\mathbf{a'}, \mathbf{b})$ runs over the set of injective maps such that $f' \circ i =f$. 
\end{nota}

Denote by $\f (\finj)$ the category of functors from $\finj$ to $\kring$-vector spaces. Then, for each $b \in \nat$, the functor $\kring^{\hfi (-, \mathbf{b})}$ is injective in $\f (\finj)$, by Yoneda's lemma. More explicitly, it corepresents the functor $F \mapsto F(\mathbf{b})^\sharp$, for $F \in \ob\f(\finj)$, where $^\sharp$ denotes vector space duality. It follows that $\{ \kring^{\hfi (-, \mathbf{b})} \ | \  b \in \nat \}$ is a set of injective cogenerators of $\f (\finj)$. 

\begin{prop}
\label{prop:iso_trans_inj}
For $b \in \nat$, 
\begin{enumerate}
\item 
$\kring \hfi (-,\mathbf{b})\trans$ is isomorphic to the injective cogenerator $\kring^{\hfi (-, \mathbf{b})}$ of $\f (\finj)$; 
\item 
$\kring \hfi (- , \mathbf{b}) \trans$ yields a functor 
$ 
\kring \hfi (- , \mathbf{b}) \trans
\ : \ 
\finj 
\rightarrow 
\kring \sym_b\dash\modules.
$
\end{enumerate} 
\end{prop}

\begin{proof}
The first statement follows by using the isomorphism of $\kring$-vector spaces 
\[
\kring X \cong \kring ^X, 
\]
for $X$ a finite set. This corresponds to the pairing $\kring X \otimes \kring X \rightarrow \kring $ sending $[x] \otimes [x']$ to $1$ if $x=x' \in X$ and zero otherwise; this is equivariant with respect to automorphisms of $X$, for the trivial action on $\kring$. The structure of $\kring^{\hfi (-, \mathbf{b})}$ corresponds to that of $\kring \hfi (-,\mathbf{b})\trans$ across these isomorphisms. 

The second statement is clear.
\end{proof}

\begin{rem}
The  description $\kring \hfi (-, \mathbf{b})\trans$ is preferred here, as opposed to  $\kring^{\hfi (-, \mathbf{b})}$, since it is generalized in \cite{P_wall} to a functor on a larger category (in which $b$ is allowed to vary), for which this formulation is more natural.
\end{rem}

For $0 < a \leq b \in \nat$, applying  $\kring \hfi (- , \mathbf{b}) \trans$ to $\iota_{a-1, a} \in \hfi (\mathbf{a-1},\mathbf{a})$ gives a 
$\sym_{a-1}\op \times \sym_b$-equivariant map
\begin{eqnarray}
\label{eqn:a-1_a}
\kring \hfi (\mathbf{a-1}, \mathbf{b} ) 
\rightarrow 
\kring \hfi (\mathbf{a}, \mathbf{b}) \downarrow^{\sym_a}_{\sym_{a-1}}.
\end{eqnarray}
where $\sym_{a-1} \subset \sym_a$ is induced by $\mathbf{a-1} \subset \mathbf{a}$.

As a $\sym_{a-1}\op \times \sym_b$-module, $\kring \hfi (\mathbf{a-1}, \mathbf{b} )$ is generated by $[\iota_{a-1,b}]$. The above map is thus determined by the image of $[\iota_{a-1,b}]$.  From the definition of $\kring \hfi (-,-)\trans$, 
one checks that the following holds:

\begin{lem}
\label{lem:iota_a-1_a}
For $0 < a \leq b \in \nat$, the morphism $\kring \hfi (\iota_{a-1,a}, \mathbf{b})$ is determined as a morphism of $\sym_{a-1}\op \times \sym_b$-modules by: 
\[
[\iota_{a-1,b}] 
\mapsto 
\sum_{z \in \mathbf{b} \backslash \mathbf{a-1}}(a,z)[\iota_{a,b}],
\] 
where the transposition $(a,z)\in \sym_b$ acts via the $\sym_b$-action on $\kring \hfi(\mathbf{a}, \mathbf{b})$.
\end{lem}

\begin{nota}
\label{nota:tr}
For $0 < a \leq b \in \nat$, denote by 
\[
\tr_{a,b} :
\kring \hfi (\mathbf{a-1}, \mathbf{b})  \uparrow _{\sym_{a-1}} ^{\sym_a}
\rightarrow 
\kring \hfi (\mathbf{a}, \mathbf{b}).
\]
the $\sym_a \op \times \sym_b$-equivariant map adjoint to (\ref{eqn:a-1_a}).
 For $a=0$, $\kring \hfi (\mathbf{a-1}, \mathbf{b})$ (and hence the map $\tr_{0,b}$ also) is taken to be $0$.
\end{nota}

\begin{rem}
For $0<a \leq b$, clearly $[\iota_{a-1,b}]$ generates $\kring \hfi (\mathbf{a-1}, \mathbf{b})  \uparrow _{\sym_{a-1}} ^{\sym_a}$ as a $\sym_a \op \times \sym_b$-module, hence $\tr_{a,b}$ is also determined by the image of $[\iota_{a-1,b}] $ given in Lemma \ref{lem:iota_a-1_a}.

\end{rem}

\section{Representation theory of the symmetric groups}
\label{sect:rep}

This Section reviews elements of the representation theory of the symmetric groups that are required later. References are given to the book \cite{MR2643487} so as to have a convenient single reference for the reader. In particular, notation and conventions follow this reference.

Here $\kring$ is taken to be a field of characteristic zero.

%%%%%%%%%%%%%%%%%%%%%%%%%%%%%%%%%%%%%%%%%%ùùù
\subsection{Background}

\begin{nota}
For $n \in \nat$,
\begin{enumerate}
\item
$\sym_n$ denotes the symmetric group on $n$ letters, which identifies with the group of  automorphisms of $\mathbf{n}$;
\item 
for $\lambda \vdash n$ a partition of $n$, $S^\lambda$ denotes the associated simple representation, indexed so that $S^{(n)}$ is the trivial representation and $S^{(1^n)}$ is the signature representation. 
\end{enumerate}
\end{nota}

The conventions used here follow \cite{MR2643487}. In particular, a partition $\lambda = (\lambda_1, \ldots , \lambda_l)$ has associated Young diagram in which the $i$th row has length $\lambda_i$. For instance, the partition $(4,2,1)$ has Young diagram:

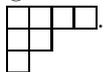
\ytableausetup{smalltableaux}

\begin{ytableau}
\  &   &  &  \\
&   \\
\ 
\end{ytableau}.

Young diagrams are given coordinates via $(\mathrm{rows}, \mathrm{columns})$. For example, in the above case:

\ytableausetup{centertableaux,boxsize={2em}}
\begin{ytableau}
\scriptstyle(1,1) & \scriptstyle(1,2)  & \scriptstyle(1,3)  & \scriptstyle(1,4) \\
\scriptstyle(2,1)  & \scriptstyle(2,2)  \\
\scriptstyle(3,1)
\end{ytableau}

\begin{nota}
Write $\preceq$ for the partial order on the set of partitions defined by $\mu \preceq \lambda$ if and only $\mu_i \leq \lambda_i$ for all $i \in \nat$ (unspecified entries are understood to be zero).  Equivalently, $\mu \preceq \lambda$ if and only if the Young diagram of $\mu$ is contained in that of $\lambda$.

If $\mu \preceq \lambda$, $\lambda/ \mu$ denotes the associated skew diagram; this is viewed as the complement of the Young diagram of $\mu$ in that of $\lambda$. 
\end{nota}

\begin{exam}
If $\mu = (1,1)$ and $\lambda = (4,2,1)$ then $\mu \preceq \lambda$, which can be represented by the diagram 
\ytableausetup{smalltableaux}

\begin{ytableau}
*(lightgray)  &  &  &  \\
*(lightgray) &   \\
\ 
\end{ytableau}
in which the Young diagram of $\mu$ is shaded.

The skew diagram $\lambda /\mu$ is represented by:
\begin{ytableau}
\none &   &  &  \\
\none &   \\
\ 
\end{ytableau}.
\end{exam}

\begin{rem}
\label{rem:partial_order_skew}
For a fixed partition $\lambda$, the partial order $\preceq$ induces a partial order on the skew partitions  such that $\lambda/\mu_1 \preceq \lambda/\mu_2$ if and only if $\mu_2 \preceq \mu_1$ (note the reversal of the order). This corresponds to the inclusion of skew diagrams. 
\end{rem}

\begin{exam}
\label{exam:skew_partitions}
For $\lambda =(4,2,1)$ and $\mu = (1,1)$, $\mu' = (2,1)$, clearly one has $\mu \preceq \mu'$, hence $\lambda/ \mu' \preceq \lambda /\mu$, which corresponds at the level of the skew diagrams to the inclusion of a sub diagram:
\begin{ytableau}
\none & \none   & &\\
\none &  \\
\ 
\end{ytableau}
$\quad \preceq \quad$ 
\begin{ytableau}
\none &   &  &  \\
\none &   \\
\ 
\end{ytableau}.
\end{exam}

\begin{defn}
\label{defn:horizontal_strip}
For partitions $\mu \preceq \lambda$, the skew partition $\lambda / \mu$ is a horizontal strip if each column of the skew diagram contains at most one box.
\end{defn}

This can be reformulated using the following. 

\begin{nota}
\label{nota:hs}
For a partition $\lambda$, let $\hs{\lambda}$ denote the partition obtained by forgetting $\lambda_1$. (In terms of the associated Young diagrams, this corresponds to removing the first row.)
\end{nota}

Clearly $\hs{\lambda} \preceq \lambda$ and $\lambda / \hs{\lambda}$ is a horizontal strip with Young diagram that contains precisely one box in each non-empty column of $\lambda$. Moreover, one has:

\begin{lem}
\label{lem:hs}
For $\mu \preceq \lambda$, $\lambda / \mu$ is a horizontal strip if and only if $\hs{\lambda} \preceq \mu$. In particular, $\lambda/ \hs{\lambda}$ is the maximal element of the poset of horizontal strips for $\lambda$.
\end{lem}

\begin{exam}
In Example \ref{exam:skew_partitions}, $\hs{\lambda} = \mu'$. Visibly $\lambda / \mu'$ is a horizontal strip, whereas $\lambda / \mu$ is not, which corresponds to the obvious fact that $\mu' \not \preceq \mu$.
\end{exam}

To each skew partition $\lambda/\mu$, one associates a skew representation as follows. Here, for $a \leq b \in \nat$, one uses $\iota_{a,b} : \mathbf{a} \subset \mathbf{b}$ to define the Young subgroup 
$
\sym_a \times \sym_{b-a} \subset \sym_b,
$ 
as usual. Recall the following:

\begin{defn}
\label{defn:skew_representation}
For $\lambda \vdash b$ and $\mu \vdash a$ with $\mu \preceq \lambda$, the skew representation $S^{\lambda/ \mu}$ is the $\sym_{b-a}$-module
\[
S^{\lambda/ \mu }:= 
\hom_{\sym_{a}}(S^\mu, S^\lambda),
\]
where the $\sym_{b-a}$ action is induced by the restriction along $\sym_{b-a} \subset \sym_b$ of the action on $S^\lambda$.
\end{defn}

The importance of the skew representations is illustrated by the following (see \cite[Proposition 3.5.5]{MR2643487}, for example.):

\begin{prop}
\label{prop:restrict_S_lambda}
For $a \leq b \in \nat$ and $\lambda \vdash b$, 
the restriction of $S^\lambda$ to  $\sym_a \times \sym_{b-a}$ identifies as:
\[
S^\lambda \downarrow^{\sym_b} _{\sym_a \times \sym_{b-a}} 
\cong 
\bigoplus_{ \substack {\mu \vdash a\\ \mu \preceq \lambda }}
  S^\mu \boxtimes S^{\lambda /\mu},
\]
where $\boxtimes$ denotes the exterior tensor product.
\end{prop}

%%%%%%%%%%%%%%%%%%%%%%%%%%%%%%%%%%%%%%%%%%%%%%%%%%%%%%ù
\subsection{Tableaux and axial distance}

Recall that, for $\lambda\vdash b$, a Young tableau of shape $\lambda$ is a bijection between the boxes of the Young diagram of $\lambda$ and the set $\mathbf{b}$.

A Young tableau $T$ of shape $\lambda$ is standard if the entries increase both along the rows and down the columns of $T$. The set of standard tableaux of shape $\lambda$ is denoted by $\tab (\lambda)$.

These notions carry over to skew diagrams of the form $\lambda/ \mu$: 
 if $\lambda \vdash a $ and $\mu \vdash b$ with $\mu \preceq \lambda$, then a tableau of shape $\lambda/\mu$ is a bijection between the boxes of the Young diagram and $\mathbf{b-a}$. A tableau of shape $\lambda/\mu$ is standard if entries increase both along the rows and down the columns. The set of standard skew tableaux of shape $\lambda / \mu$ is denote $\tab (\lambda/\mu)$. (Thus, on taking $\mu=(0)$, one recovers $\tab (\lambda)$.)  
  
For $\lambda/ \mu$ a horizontal strip, one has the  distinguished standard skew-tableau:

\begin{nota}
\label{nota:trev}
For $\lambda / \mu$ a horizontal strip, let $T\rev \in \tab (\lambda/ \mu)$ be the standard Young tableau of shape $\lambda/ \mu$ such that numbers increase from left to right.
\end{nota}

\begin{exam}
\label{exam:trev}
Take $\lambda =(5,2,1,1)$ and $\mu= (3,1,1)$, so that $\hs{\lambda} \preceq \mu \preceq \lambda$; then for $\lambda /\mu$,  the tableau $T\rev$ is:
 \ytableausetup{smalltableaux}
\begin{ytableau}
\none & \none & \none &\none & 3 & 4
\\
\none & \none & 2
\\
\none
\\
\none & 1 
\end{ytableau}. 
\end{exam}

We introduce the following:

\begin{nota}
\label{nota:coordinates}
For $i \in \mathbf{b}$ and a Young tableau of shape $\lambda$, denote by $(u_i, v_i)$ the coordinates of $i$.
\end{nota}

Then, given a tableau $T$ of shape $\lambda/ \mu$  one has the notion of axial distance between entries of the tableau:

\begin{defn}
\label{defn:axial}
For $T$ a tableau of shape $\lambda / \mu$, where $\lambda \vdash b$ and $\mu \vdash a$ with $\mu \preceq \lambda$:
\begin{enumerate}
\item 
for $i,j \in \mathbf{b-a}$, the axial distance $a(j,i) = a (j,i)(T)$ from $j$ to $i$ in $T$ is  $(v_{j}-v_i) -  (u_{j} -u_i) = (v_j -u_j) - (v_i-u_i)$;
\item 
for $j \in \mathbf{b-a-1}$, $r_j=r_j(T):= a(j+1,j)(T)$, the axial distance from $j+1$ to $j$ in $T$.  
\end{enumerate} 
\end{defn}

Axial distance satisfies the following evident additivity property:

\begin{lem}
\label{lem:axial_additivity}
For $T$ a tableau of shape $\lambda / \mu$, where $\lambda \vdash b$ and $\mu \vdash a$ with $\mu \preceq \lambda$, and for  $i,j,k \in \mathbf{b-a}$:
\[
a (k,i) (T) = a(k,j) (T) + a (j,i) (T).
\]
In particular, 
 $
a(3,1) (T) = r_1 (T) + r_2(T).
$
\end{lem}

Axial distance gives the following criterion for a skew diagram $\lambda/ \mu$ to be a horizontal strip:

\begin{lem}
\label{lem:Trev}
Suppose $\lambda \vdash b$ and $\mu \vdash a$ with  $\mu \preceq \lambda$. Then $\lambda/ \mu$ is a horizontal strip if and only if there exists $T \in \tab(\lambda/ \mu)$ such that, for all $1 \leq i<j \leq b-a$, the axial distance  $a (j,i)(T)$ from $j$ to $i$ is positive. 

Moreover, such a  standard Young tableau is unique, namely is $T\rev$.
\end{lem}

\begin{proof}
If $\lambda /\mu$ is not a horizontal strip, there exists a column containing at least two boxes. For $T \in \tab(\lambda/\mu)$, consider such adjacent boxes with labels
 \ytableausetup{smalltableaux}
\begin{ytableau}
i \\
j
\end{ytableau}.
Since $T$ is standard, $i<j$ and $a(j,i)(T)=-1$. 

If $\lambda/ \mu$ is a horizontal strip, then one checks that $T\rev$  is the unique element of $\tab (\lambda /\mu)$ with the required property.
\end{proof}

%%%%%%%%%%%%%%%%%%%%%%%%%%%%%%%%%%%%%%%%%%%%%%%%%%%%%%%%%%%%%%%%
\subsection{Young's orthogonal form}

Since, for $\lambda \vdash b$, a tableau of shape $\lambda$ is an isomorphism from the set of boxes of the Young diagram of $\lambda$ to $\mathbf{b}$, the symmetric group $\sym_b$ acts by post-composition. There is a subtlety, if $T \in \tab(\lambda)$  and $\sigma \in \sym_b$, $\sigma T$ is not necessarily standard. 

We focus upon the Coxeter generators of the symmetric groups:
 
\begin{nota}
\label{nota:coxeter}
For $n \in \nat$, the set of Coxeter generators of $\sym_n$ is  the set of transpositions $s_i:= (i, i+1)$, for $1\leq i < n$. (The value of $n$ is usually understood, so is not included in the notation.)
\end{nota}

If $T \in \tab (\lambda)$ and $1 \leq j <b$, then $s_j T$ is standard unless the entries $j$, $j+1$ occur in adjacent boxes 
\ytableausetup{centertableaux}
\begin{ytableau}
\ & \ 
\end{ytableau} or 
\begin{ytableau}
\ \\
\ 
\end{ytableau}; 
 equivalently $|r_j|= |r_j(T)|=1$.

In order to use Young's orthogonal form, we work over $\kring = \mathbb{R}$. For $\lambda \vdash b$  
(see \cite[Theorem 3.4.4]{MR2643487}) Young's orthonormal basis of $S^\lambda$ is
$$
\{ w_T | T \in \tab (\lambda) \}
$$
with the $\sym_b$-action determined by the following action of the Coxeter generators:
\begin{eqnarray*}
s_j w_T &=& \frac{1}{r_j} w_T + \sqrt{ 1 - \frac{1}{r_j^2}}\  w_{s_jT}
\\
s_j w_{s_jT}&=& - \frac{1}{r_j} w_{s_jT} + \sqrt{ 1 - \frac{1}{r_j^2}} \ w_{T},
\end{eqnarray*}
(assuming that $s_jT$ is standard for the second equation).

As above, $s_jT$ is non-standard if and only if  $|r_j|=1$ (when $\frac{1}{r_j}= r_j$), in which case the first equation reduces to 
\begin{eqnarray*}
s_j w_T &=& r_j w_T. 
\end{eqnarray*}

The Young orthonormal basis adapts to working with skew diagrams (see \cite[Theorem 3.5.5]{MR2643487}). For $\lambda \vdash b$ and $\mu \vdash a$ with $\mu \preceq \lambda$, the skew representation $S^{\lambda/\mu}$ has orthonormal basis 
\[
\{ \Phi _T | T \in \tab (\lambda/ \mu) \}
\]
where the Coxeter generators of $\sym_{b-a}$ act by 
\begin{eqnarray*}
s_j \Phi_T &=& \frac{1}{r_j} \Phi_T + \sqrt{ 1 - \frac{1}{r_j^2}}\  \Phi_{s_jT}.
\end{eqnarray*}

\begin{rem}
If $\lambda/\mu$ is a horizontal strip, then $s_jT$ is not standard if and only if $r_j=1$, in which case the above action reduces to $s_j \Phi_T = \Phi_T$.
\end{rem}

%%%%%%%%%%%%%%%%%%%%%%%%%%%%%%%%%%%%%%%%%%%%%%%%%%%%%%%%%%%%%%%%%
\subsection{First applications}
\label{subsect:first_apps}

For partitions $\lambda, \lambda'$ of $b$, since the contragredient of $S^\lambda$ is isomorphic to $S^\lambda$, one has
\begin{eqnarray}
\label{eqn:invt_tensor_product_simples}
(S^\lambda \otimes S^{\lambda'})^{\sym_b} \cong \left\{ 
\begin{array}{ll}
\kring & \lambda = \lambda' \\
0 & \mbox{otherwise}
\end{array}
\right.
\end{eqnarray}
where $\sym_b$ acts diagonally. The Young orthonormal basis allows one to identify an explicit generator of these invariants in the case $\lambda = \lambda'$:

\begin{lem}
\label{lem:invt_prod_tens}
For $\lambda \vdash b$,
the element $\sum_{T \in \tab(\lambda)} w_T \otimes w_T \in S^\lambda \otimes S^\lambda$ is a  generator of $(S^\lambda \otimes S^\lambda )^{\sym_b}= \kring$. 
\end{lem}

\begin{proof}
It is sufficient to check that this element is invariant under the (diagonal) action of the Coxeter generators. This is shown using the explicit action given above. 
\end{proof}

Likewise, if $\mu \vdash a$ such that $\mu \preceq \lambda$, then 
\[
(S^{\lambda /\mu})^{\sym_{b-a}}
\cong 
\left\{ 
\begin{array}{ll}
\kring & \mbox{$\lambda/ \mu$ is a horizontal strip}
\\
0 &\mbox{ otherwise,}
\end{array}
\right.
\]
(see \cite[Proposition 3.5.12]{MR2643487}, for example) and, moreover, if $\lambda/ \mu$ is a horizontal strip, then $S^{\lambda / \mu}$ is a permutation module. More precisely (see \cite[Proposition 3.5.8]{MR2643487}, for example), if $\lambda/\mu$ is a horizontal strip with rows of lengths $b_1, \ldots , b_t$ (so that $\sum_i b_i=b-a$), then 
\[
S^{\lambda / \mu} \cong \kring\uparrow_{\prod _i \sym_{b_i}} ^{\sym_{b-a}}.  
\]

This can be made explicit using the standard tableau $T\rev \in \tab (\lambda/\mu)$ (see Notation \ref{nota:trev}):

\begin{prop}
\label{prop:permutation}
For $\lambda \vdash b$, $\mu \vdash a$ such that $\hs{\lambda} \preceq \mu \preceq \lambda$, the map $1 \mapsto \Phi_{T\rev}$ induces an isomorphism of $\sym_{b-a}$-modules
\[
\sym_{b-a} \otimes_{\prod _i \sym_{b_i}} \kring 
\stackrel{\cong}{\longrightarrow} 
S^{\lambda/\mu}.
\]
\end{prop}

\begin{proof}
The hypothesis implies that $\lambda/ \mu$ is a horizontal strip.

Using the action of the Coxeter generators, one sees  that the Young subgroup $\prod _i \sym_{b_i} \subset \sym_{b-a}$ stabilizes $\Phi_{T\rev}$, hence one obtains the given morphism of $\sym_b$-modules. This is surjective, as is seen for example by applying \cite[Proposition 3.5.6]{MR2643487}. Hence it is an isomorphism, for dimension reasons.
\end{proof}

Moreover, the Young orthonormal basis allows an explicit generator for the invariants $(S^{\lambda/\mu}) ^{\sym_{b-a}}$ to be exhibited:

\begin{prop}
\label{prop:invt_skew}
If $\hs{\lambda}\preceq \mu \preceq \lambda$, then $(S^{\lambda/\mu}) ^{\sym_{b-a}} \cong \kring$ is generated by the element of the form 
\[
\sum_{T \in \tab(\lambda/ \mu)} \beta_T \Phi_T 
\]
that is uniquely determined by $\beta_{T\rev}=1$. 

The coefficients are determined by the following: if $T', T'' \in \tab(\lambda/\mu)$ such that $T''= s_i T'$, then 
\[
\frac{\beta_{T''}}{\beta_{T'}} 
=  \sqrt{\frac{r_i-1}{r_i+1}},
\]
where $r_i= r_i(T')$, which satisfies $r_i \not \in \{-1, 0,1\}$, by the hypothesis on $T',T''$. In particular,  $\beta_T \in \kring^\times$ for all $T \in \tab (\lambda)$.  
\end{prop}

\begin{proof}
Under the hypothesis upon $\lambda/\mu$,  $(S^{\lambda/\mu}) ^{\sym_{b-a}}= \kring$, hence there exists an invariant element of the form $
\sum_{T \in \tab(\lambda/ \mu)} \beta_T \Phi_T$ with not all coefficients zero. Using the action of the Coxeter elements, as explained below, one deduces that {\em all} coefficients must be non-zero. The expression can then be normalized by imposing $\beta_{T\rev}=1$, so that the $\beta_T$ are uniquely determined.

To establish the explicit relation between $\beta_{T'}$ and $\beta_{T''}$, we use that invariance requires that the coefficient of $\Phi_{T'}$ in $
s_i \big( \sum_{T \in \tab(\lambda/ \mu)} \beta_T \Phi_T \big) 
$ must be equal to $\beta_{T'}$. This coefficient is equal to that  of $\Phi_{T'}$ in 
$
s_i (\beta_{T'}\Phi_{T'} + \beta_{T''} \Phi_{T''})
$
which is $\beta_{T'}\frac{1}{r_i} + \beta_{T''}\sqrt{ 1 - \frac{1}{r_i^2}}$. The stated relation follows from the equality 
$\beta_{T'}=\beta_{T'}\frac{1}{r_i} + \beta_{T''}\sqrt{ 1 - \frac{1}{r_i^2}}$.

The rational number $\frac{r_i-1}{r_i+1}$ is defined and strictly positive, since the integer $r_i$ does not belong to $\{ -1, 0, 1 \}$.

This gives the calculation of all of the coefficients $\beta_{T}$ recursively, starting from $\beta_{T\rev}=1$ since the equivalence relation on $\tab(\lambda/\mu)$ generated by $T' \sim T''$  if there exists a Coxeter generator $s_i$ such that $T''= s_i T'$ has a single equivalence class (see \cite[Corollary 3.1.6]{MR2643487}, for example). In particular, one sees that all of the coefficients $\beta_T$ are non-zero.
\end{proof}

\begin{rem}
\label{rem:permutation_versus_Phi}
Comparing Propositions \ref{prop:permutation} and \ref{prop:invt_skew} highlights one of the subtleties here. When $\lambda/ \mu$ is a horizontal strip, $S^{\lambda /\mu}$ is a permutation module with permutation basis 
$ 
\sigma \Phi_{T\rev} 
$ as $\sigma$ ranges over a set of coset representatives for $\sym_{b-a}/ \prod _i \sym_{b_i}$.

However, the Young orthonormal basis of $S^{\lambda/\mu}$ is clearly not a permutation basis (except in degenerate cases). This is witnessed by the appearance of  coefficients $\beta_T \neq 1$ in Proposition \ref{prop:invt_skew}.
\end{rem}

\section{A first analysis of $\kring \hfi (\mathbf{a}, \mathbf{b})$} 
\label{sect:perm}

In this Section, $\kring$ is a field of characteristic zero. In particular, the categories of representations that intervene are all semisimple. 

%%%%%%%%%%%%%%%%%%%%%%%%%%%%%%%%%%%% 
\subsection{The permutation representation}

For $a,b \in \nat$, $\kring \hfi (\mathbf{a}, \mathbf{b})$ is a $ \sym_a\op\times \sym_b$-module. It is  the permutation module associated to the transitive $ \sym_a\op \times \sym_b$-set $\hfi (\mathbf{a}, \mathbf{b})$; in particular, it is non-zero if and only if $a\leq b$. If $a \leq b$, a generator  is given by the canonical inclusion $\iota_{a,b} : \mathbf{a}= \{1, \ldots ,a \} \subset \mathbf{b} = \{1, \ldots , b\}$.

\begin{rem}
\label{rem:op}
The inverse $g \mapsto g^{-1}$ gives the isomorphism of groups $\sym_a \cong \sym_a \op$, so that  $\hfi (\mathbf{a}, \mathbf{b})$ can be considered as a  $\sym_a \times \sym_b$-set and  $\kring \hfi (\mathbf{a}, \mathbf{b})$  as a  $\sym_a \times \sym_b$-module.
\end{rem}

The  inclusion $\iota_{a,b}$ determines $\sym_a \subset \sym_b$ and hence  the Young subgroup $\sym_a \times \sym_{b-a} \subset \sym_b$.

\begin{lem}
\label{lem:perm}
For $a \leq b \in \nat$, considered as a  $\sym_b \times \sym_s$-set, there is an isomorphism
 $$
 \hfi (\mathbf{a}, \mathbf{b})\cong (\sym_b \times \sym_a) / (\sym_{b-a} \times \Delta \sym_a),
 $$
  where $\Delta \sym_a \subset \sym_b \times \sym_a$ is the diagonal inclusion. 

Hence there is an isomorphism of $\sym_a\op \times \sym_b$-sets
$$
\hfi (\mathbf{a}, \mathbf{b})
\cong 
\sym_b / \sym_{b-a}
$$
with the regular left $\sym_b$-action  and $\sym_a $ acting via right multiplication on $\sym_b$; i.e., for $g, g' \in \sym_b $ and $h \in \sym_a$, $g [g'] h = [gg'h]$. 
\end{lem}

\begin{proof} The stabilizer of $\iota_{a,b}$ for the left  $\sym_b \times \sym_a$-action on $\hfi (\mathbf{a}, \mathbf{b})$ is the subgroup $(\sym_{b-a} \times \Delta \sym_a)$, which gives the first statement. 

As a left $\sym_b$-set, the quotient by the action of the diagonal subgroup  $\Delta \sym_a$ clearly identifies as $\sym_b / \sym_{b-a}$. It remains to consider the $\sym_a$-action. Considered now as a right action, one checks that this is as given. 
(Note that this is well-defined, since the right action of $\sym_a$ on $\sym_b$ commutes with that of $\sym_{b-a}$, due to the inclusion $\sym_a \times \sym_{b-a} \subset \sym_b$.) 
\end{proof}

We next identify the composition factors of the permutation module $\kring \hfi (\mathbf{a}, \mathbf{b})$. For this, recall that the set of isomorphism classes of simple $\sym_b \times \sym_a$-modules is indexed by pairs of partitions $(\lambda \vdash b, \nu \vdash a)$, corresponding to $S^\lambda \boxtimes S^\nu$. 

Recall the notation $\hs{\lambda}$ introduced in Notation \ref{nota:hs}, which has the property that, for $\mu \preceq \lambda$, $\lambda/\mu$ is a horizontal strip if and only if $\hs{\lambda} \preceq \mu$.

\begin{prop}
\label{prop:inj_comp_fact}
For $a \leq b \in \nat$, there is an isomorphism of  $\sym_b \times \sym_a$-modules 
$$
\kring \hfi (\mathbf{a}, \mathbf{b})
\cong 
\bigoplus_{ \substack {\lambda \vdash b, \nu \vdash a\\ \hs{\lambda} \preceq \nu \preceq \lambda  }}
S^\lambda \boxtimes S^\nu.
$$

In particular, the $\sym_b \times \sym_a$-module $\kring \hfi (\mathbf{a}, \mathbf{b})$ is multiplicity-free.
\end{prop}

\begin{proof}
To prove the result, it suffices to show that 
\[
\hom_{\sym_b \times \sym_a} ( \kring \hfi (\mathbf{a}, \mathbf{b}), S^\lambda \boxtimes S^\nu) 
= 
\left\{ 
\begin{array}{ll}
\kring & \lambda \vdash b, \  \nu \vdash a, \ \hs{\lambda} \preceq \nu \preceq \lambda
\\
0 & 
\mbox{otherwise.}
\end{array}
\right.
\]

By Lemma \ref{lem:perm},  $\hom_{\sym_b \times \sym_a} ( \kring \hfi (\mathbf{a}, \mathbf{b}), S^\lambda \boxtimes S^\nu)
\cong 
(S^\lambda \boxtimes S^\nu)^{\sym_{b-a} \times \Delta \sym_a}$. The right hand side only depends on the restriction of the $\sym_b$-action on $S^\lambda$ to $\sym_{b-a} \times \sym_a$ and, by Proposition \ref{prop:restrict_S_lambda},  identifies with 
\[
\big(\bigoplus_{ \substack {\mu \vdash a\\ \mu \preceq \lambda }}
S^{\lambda /\mu} \boxtimes (S^\mu \otimes S^\nu)\big)^{\sym_{b-a} \times  \sym_a},
\]
where $\sym_{b-a}$ acts on $S^{\lambda /\mu}$ and $\sym_a$ acts diagonally on $S^\mu \otimes S^\nu$.  This identifies as:
\[
\bigoplus_{ \substack {\mu \vdash a\\ \mu \preceq \lambda }}
(S^{\lambda /\mu})^{\sym_{b-a}} \otimes (S^\mu \otimes S^\nu)^{\sym_a}.
\]

Now, as recalled in Section \ref{subsect:first_apps},
\[
(S^{\lambda /\mu})^{\sym_{b-a}}
\cong 
\left\{ 
\begin{array}{ll}
\kring & \hs{\lambda}\preceq \mu \preceq \lambda
\\
0 &\mbox{ otherwise,}
\end{array}
\right.
\]
 and, as in equation (\ref{eqn:invt_tensor_product_simples}), 
$$(S^\mu \otimes S^\nu)^{\sym_a}=
\left\{
\begin{array}{ll}
\kring & \mu = \nu \\
0 & \mbox{ otherwise.}
\end{array}
\right.
$$
The result follows.
\end{proof}

%%%%%%%%%%%%%%%%%%%%%%%%%%%%%%%%%%%%%%%%%%%%%%%%%%%%%%%%%%%%%%%
\subsection{An explicit map}

In this subsection, Proposition \ref{prop:inj_comp_fact} is refined by exhibiting an explicit non-trivial map from $ \kring \hfi (\mathbf{a}, \mathbf{b})$ to $S^\lambda \boxtimes S^\nu$.

\begin{nota}
\label{nota:tab_relative}
For $\nu \preceq \lambda$, let $\tab(\lambda; \nu) \subset \tab(\lambda)$ be the set of standard tableaux $T$ for which the restriction $T|_\nu$ to the diagram $\nu$ belongs to $\tab (\nu)$ (i.e., contains only the numbers $\{ 1, \ldots , |\nu|\}$). 
\end{nota}

The following is clear:

\begin{lem}
\label{lem:tab_relative}
For $\nu \preceq \lambda$, the map 
\begin{eqnarray*}
\tab(\lambda; \nu)
&\rightarrow & 
\tab (\nu) \times \tab (\lambda/\nu)
\\
T &\mapsto &(T|_\nu, T|_{\lambda/\nu})
\end{eqnarray*}
 is a bijection, 
 where the restriction $T|_{\lambda/ \nu}$ to the skew diagram $\lambda /\nu$ is considered as an element of $\tab (\lambda/\nu)$ by relabelling (i.e., subtracting $|\nu|$ from the entries).
 \end{lem}

\begin{prop}
\label{prop:X}
For $\lambda \vdash b$, $\nu \vdash a$ such that $\hs{\lambda} \preceq \nu \preceq \lambda$, the element 
\[
X_{\lambda, \nu}:= 
\sum _{T \in \tab(\lambda; \nu)}
\beta_{T|_{\lambda/\nu}} w_T \otimes w_{T|_\nu} 
\in S^\lambda \otimes S^\nu
\]
is a generator of $(S^\lambda \boxtimes S^\nu)^{\sym_{b-a} \times \Delta \sym_a}
\cong 
\kring$, where the coefficients $\beta_*$ are given by Proposition \ref{prop:invt_skew}.

Hence a generator of $\hom _{\sym_b \times \sym_a} (\kring \hfi (\mathbf{a}, \mathbf{b}), S^\lambda \boxtimes S^\nu)$ is given by
$\iota_{a,b} \mapsto   X_{\lambda, \nu}$.
\end{prop}

\begin{proof}
Suppose that $T \in \tab(\lambda; \nu)$; then restricting to $\sym_{b-a} \times \sym_a$ and exploiting the bijection of Lemma \ref{lem:tab_relative}, the action of the relevant Coxeter generators on $w_T$ corresponds to that on 
\[
\Phi_{T|_{\lambda/\nu}} \otimes w_{T|_\nu}
\]
where $\sym_{b-a}$ acts on the first factor and $\sym_a$ on the second.  
The first statement then follows by putting together Lemma \ref{lem:invt_prod_tens} and  Proposition \ref{prop:invt_skew}.

The second statement then follows from Proposition \ref{prop:inj_comp_fact} and its proof.
\end{proof}

%%%%%%%%%%%%%%%%%%%%%%%%%%%%%%%%%%%%%%%%%%%%%%%%%%%%%%%%%%%%%%%%
\subsection{Comparing representations}

Our ultimate goal is to understand the cokernel of 
\[
\tr_{a,b} : 
\kring \hfi (\mathbf{a-1 }, \mathbf{b})\uparrow_{\sym_{a-1}} ^{\sym_a}
\rightarrow 
\kring \hfi (\mathbf{a}, \mathbf{b})
.
\]
In this subsection, we only compare these representations.

If $1 \leq a \leq b$, then Proposition \ref{prop:inj_comp_fact} applies also to $\kring \hfi (\mathbf{a-1}, \mathbf{b})$, considered as a $\sym_b \times \sym_{a-1}$-module. 
 One can then induce up to a $\sym_b \times \sym_a$-module, giving the following  Corollary, in which the second isomorphism is given by Pieri's rule:

\begin{cor}
\label{cor:restrict_naive}
For $1 \leq a \leq b \in \nat$, there is an isomorphism of  $\sym_b \times \sym_a$-modules 
$$
\kring \hfi (\mathbf{a-1 }, \mathbf{b})\uparrow_{\sym_{a-1}} ^{\sym_a}
\cong 
\bigoplus_{ \substack {\lambda \vdash b, \kappa \vdash a-1\\\hs{\lambda}\preceq \kappa \preceq \lambda  }}
S^\lambda \boxtimes (S^\kappa \uparrow_{\sym_{a-1}} ^{\sym_a})
\cong 
\bigoplus_{ \substack {\lambda \vdash b,   \kappa \vdash a-1\\ \hs{\lambda}\preceq  \kappa \preceq \lambda  }}
\bigoplus_{ \substack {\mu \vdash a \\ \kappa \preceq \mu }}
S^\lambda \boxtimes S^\mu .
$$
\end{cor}

 As an immediate  consequence, one has:

\begin{lem}
\label{lem:comp_factor_max_hor_strip}
For $a \leq b \in \nat$, a composition factor $S^\lambda \boxtimes S^\nu$ of $\kring \hfi (\mathbf{a}, \mathbf{b})$ does not occur in $\kring \hfi (\mathbf{a-1 }, \mathbf{b})\uparrow_{\sym_{a-1}} ^{\sym_a}$ if and only if 
 $\nu  = \hs{\lambda}$ and $\lambda_1 = b-a$.

Hence, there is a bijection between the  set of such composition factors and both of the following:
\begin{enumerate}
\item 
$\{ \lambda \vdash b \  | \ \lambda_1 = b-a\}$;
\item 
$\{ \nu \vdash a \ | \ \nu_1 \leq b-a\}$.
\end{enumerate}
\end{lem}

\begin{proof}
Suppose that $S^\lambda \boxtimes S^\nu$ is a composition factor of $\kring \hfi (\mathbf{a}, \mathbf{b})$ and that $\lambda/\nu$ is not the maximal horizontal strip (i.e., $\hs {\lambda} \preceq \nu$ is not an equality). Then consider the rightmost non-empty column of $\lambda$ in which $\lambda/\nu$ is empty; such a column exists by the hypothesis. Let $\kappa \vdash a-1$ be the unique partition obtained from $\nu$ by decreasing the length of this column by one (the verification that $\kappa $ is a partition is left to the reader).

By construction, $\lambda/\kappa$ is a horizontal strip and $S^\kappa \uparrow_{\sym_{a-1}} ^{\sym_a}$ contains $S^\nu$ as a composition factor, by Pieri's rule. It follows from Corollary \ref{cor:restrict_naive} that $S^\lambda \boxtimes S^\nu$ is a composition factor of $\kring \hfi (\mathbf{a-1 }, \mathbf{b})\uparrow_{\sym_{a-1}} ^{\sym_a}$.

If  $\nu = \hs{\lambda}$, then the composition factor $S^\lambda \boxtimes S^\nu$ cannot arise in $\kring \hfi (\mathbf{a-1 }, \mathbf{b})\uparrow_{\sym_{a-1}} ^{\sym_a}$. 
\end{proof}

\section{Calculating the cokernel of $\tr_{a,b}$}
\label{sect:calc}

Fix $1 \leq a \leq b$; our aim is to determine
the cokernel of 
\[
\tr_{a,b} : 
\kring \hfi (\mathbf{a-1 }, \mathbf{b})\uparrow_{\sym_{a-1}} ^{\sym_a}
\rightarrow 
\kring \hfi (\mathbf{a}, \mathbf{b})
\]
as a $\sym_a \op \times \sym_b$-module.

Recall  that, for $\nu \preceq \lambda$, $\lambda/\nu$ is a horizontal strip if and only if $\hs{\lambda} \preceq \nu$. The following strengthens Lemma \ref{lem:comp_factor_max_hor_strip}  to describe $\tr_{a,b}$:

\begin{thm}
\label{thm:coker}
For $1 \leq a\leq b$, there is an isomorphism of $\sym_b\times \sym_a$-modules:
\[
\mathrm{Coker} \tr_{a,b}
\cong 
\bigoplus_{\substack{\lambda \vdash b \\ \lambda_1 =b-a}}
S^\lambda \boxtimes S^{\hs{\lambda}}.
\]
\end{thm}

The proof of this result occupies the whole of this Section.

%%%%%%%%%%%%%%%%%%%%%%%%%%%%%%%%%%%%%%%%%%%
\subsection{A first reduction}
\label{subsect:first_reduct}

By Lemma \ref{lem:comp_factor_max_hor_strip} and Proposition \ref{prop:X}, to prove the Theorem it is sufficient to prove the following statement:

\begin{prop}
\label{prop:reduction_X}
Suppose that $\lambda \vdash b$ with $\lambda_1> b-a$ and $\nu \vdash a$ such that  $\hs{\lambda}\preceq \nu \preceq \lambda$. Then 
\[
\sum _{i \in \mathbf{b} \backslash \mathbf{a-1}} 
(a, i)  X_{\lambda, \nu} \neq 0,
\]
where $(a,i) \in \sym_{b}$.
\end{prop}

\begin{rem}
The hypothesis $\lambda_1> b-a$ is equivalent to the assertion that $\nu \neq \hs{\lambda}$. 
\end{rem}

Reasoning as in the proof of Lemma \ref{lem:comp_factor_max_hor_strip}, one has:

\begin{lem}
\label{lem:kappa_from_nu}
Suppose that $\lambda \vdash b$ with $\lambda_1> b-a$ and $\nu \vdash a$ such that  $\hs{\lambda}\preceq \nu \preceq \lambda$. Then there exists $\kappa \vdash a-1$ such that $\hs{\lambda} \preceq \kappa \preceq \nu \preceq \lambda$. In particular, the skew diagram $\lambda/\kappa$ is a horizontal strip. 
\end{lem}

\begin{hyp}
\label{hyp:kappa_nu_lambda}
Henceforth in this Section, we fix $\kappa, \nu , \lambda$ as in Lemma \ref{lem:kappa_from_nu}, so that $\lambda \vdash b$ such that $\lambda_1 > b-a$, $\nu \vdash a$, $\kappa \vdash a-1$ with 
 $
\hs{\lambda} \preceq \kappa \preceq \nu \preceq \lambda.
$ 
\end{hyp}

\begin{nota}
\label{nota:skew_relative}
Denote by $\tab(\lambda/\kappa; \nu/ \kappa) \subset \tab(\lambda/\kappa)$ the set of standard skew tableaux such that the restriction to $\nu/\kappa$ is standard (i.e., the single box is labelled by $1$). 
\end{nota}

The following is clear (cf. Lemma \ref{lem:tab_relative}):

\begin{lem}
Restriction and relabelling yield a bijection:
\begin{eqnarray*}
\tab(\lambda/\kappa; \nu/ \kappa) &\stackrel{\cong}{\rightarrow} & \tab (\lambda / \nu) \\
T & \mapsto & T|_{\lambda/\nu}.
\end{eqnarray*}
\end{lem}

\begin{nota}
\label{nota:skew}
\ 
\begin{enumerate}
\item
Let $\tspec \in \tab(\lambda/\kappa; \nu/ \kappa)$ be the unique standard tableau such that $\mathbb{T}|_{\lambda/\nu}= T\rev \in \tab (\lambda/\nu)$.
\item 
Set 
\[
Y_{\lambda,\nu, \kappa}:= 
\sum_{T \in \tab(\lambda/\kappa; \nu/\kappa) }
\beta_{T|_{\lambda/\nu}} \Phi_T 
\in S^{\lambda/\kappa},
\]
\end{enumerate}
\end{nota}

The following is a consequence of Proposition \ref{prop:invt_skew}; it explains the significance of $Y_{\lambda,\nu,\kappa}$.

\begin{lem}
\label{lem:Y_invariance}
The element $Y_{\lambda,\nu,\kappa}$ is invariant under the action of $\sym_{b-a} \subset \sym_{b-a+1}$ corresponding to the inclusion $\{2, \ldots , b-a+1 \} \subset \mathbf{b-a+1}$. 
\end{lem}

Below we consider $S^{\lambda/\kappa}$  as a representation of $\mathrm{Aut}(\mathbf{b-a +1})$ rather than of the subgroup $\sym_{b-a+1}\subset \sym_b$ used above. Under this convention, the element 
$\sum_{i \in \mathbf{b} \backslash \mathbf{a-1}} (a, i) \in \sym_b$ appearing in Proposition \ref{prop:X} is reindexed to give  $\sum_{i \in \mathbf{b-a +1}} (1, i) \in \sym_{b-a+1}$. 
 
The following  reduces the proof of Theorem \ref{thm:coker} to establishing a property of the skew representation $S^{\lambda/\nu}$:
  
 \begin{prop}
 \label{prop:reduction_Y}
Under Hypothesis \ref{hyp:kappa_nu_lambda},
 \[
 \sum_{i \in \mathbf{b-a +1}} (1, i) Y_{\lambda,\nu, \kappa} \neq 0.
 \]
 \end{prop}
 
\begin{proof}[Proof of Theorem \ref{thm:coker} assuming Proposition \ref{prop:reduction_Y}]

It suffices to show that Proposition \ref{prop:reduction_Y} implies  Proposition \ref{prop:reduction_X}.
Consider
$$
\sum _{i \in \mathbf{b} \backslash \mathbf{a-1}} 
(a, i)  X_{\lambda, \nu}
= 
\sum _{T \in \tab(\lambda; \nu)}
\Big(
\sum _{i \in \mathbf{b} \backslash \mathbf{a-1}} 
\beta_{T|_{\lambda/\nu}} (a,i)w_T 
\Big)
\otimes w_{T|_\nu}.
$$

The hypothesis on $\kappa$ ensures that there exists a tableau $\tilde{T}  \in \tab(\nu)$ such that the box of $\nu/\kappa$ is labelled by $a$.  Now we may restrict to considering  the 
$T \in \tab(\lambda; \nu)$ such that $T|_\nu = \tilde{T}$; such a tableau is determined by  the standard skew tableau $T|_{\lambda/\nu}$. 

To prove Proposition \ref{prop:reduction_X}, it suffices to show that the corresponding term 
 $$
\sum_{\substack{T \in \tab(\lambda;\nu) \\ T|_\nu = \tilde{T}}}
 \sum _{i \in \mathbf{b} \backslash \mathbf{a-1}} 
\beta_{T|_{\lambda/\nu}} (a,i)w_T
=
\sum _{i \in \mathbf{b} \backslash \mathbf{a-1}}
(a,i)
\Big(
\sum_{\substack{T \in \tab(\lambda;\nu) \\ T|_\nu = \tilde{T}}}
\beta_{T|_{\lambda/\nu}}w_T
\Big)
 $$
is non-zero. 

The summation over $T \in \tab(\lambda;\nu)$ such that $T|_\nu = \tilde{T}$ is equivalent to summing over $T' \in \tab (\lambda/\kappa; \nu /\kappa)$ and one has the correspondence $w_T \leftrightarrow \Phi_{T'}$, compatibly with the respective actions of the symmetric groups. Thus,  after reindexing and by definition of $Y_{\lambda,\nu,\kappa}$, Proposition \ref{prop:reduction_Y} gives that the right hand side is non-zero, as required.
\end{proof} 
 
\begin{rem}
\ 
\begin{enumerate}
\item 
A straightforward calculation using Lemma \ref{lem:Y_invariance} shows that 
 \[
 \sum_{i \in \mathbf{b-a +1}} (1, i) Y_{\lambda,\nu, \kappa}
 \]
 lies in $(S^{\lambda/\kappa})^{\sym_{b-a+1}}$, i.e., is fully invariant.
 Since $(S^{\lambda/\kappa})^{\sym_{b-a+1}} = \kring$, establishing non-triviality of this element is equivalent to comparing it with the generator given in Proposition \ref{prop:invt_skew}. 
\item
As in Proposition \ref{prop:permutation}, $S^{\lambda/\kappa}$ is a permutation representation. The difficulty in proving Proposition \ref{prop:reduction_Y} is that the element $Y_{\lambda, \nu , \kappa}$ is not defined in terms of the permutation basis. (Cf. Remark \ref{rem:permutation_versus_Phi}.)
\end{enumerate}
\end{rem}

 %%%%%%%%%%%%%%%%%%%%%%%%%%%%%%%%%%%%%%%%%%%%%%%%%%%%%%%%%%%%%%%%%%%
 \subsection{A further reduction}
Recall the skew tableau $\tspec \in \tab (\lambda/\kappa; \nu/\kappa) \subset \tab (\lambda/\kappa)$ introduced in Notation \ref{nota:skew}. We work with the Young orthonormal basis $\{ \Phi_T | T \in \tab (\lambda/\kappa)$ for $S^{\lambda/\kappa}$.

\begin{nota}
\label{nota:coeff}
Let $\coeff (\lambda, \nu, \kappa) \in \mathbb{R}$ denote the coefficient of $\Phi_\tspec$ in $\sum_{i \in \mathbf{b-a +1}} (1, i) Y_{\lambda,\nu, \kappa}$.
\end{nota}

  Proposition \ref{prop:reduction_Y} clearly follows from:
  
 \begin{prop}
 \label{prop:coeff} 
 Under Hypothesis \ref{hyp:kappa_nu_lambda},
  $
  \coeff  (\lambda, \nu, \kappa)> 0.
  $
 \end{prop}
 
The proof of this result (and hence of Theorem \ref{thm:coker}) will be given in the next section. 

\begin{rem}
The following example shows that $\coeff (\lambda,\nu, \kappa)$ can be an arbitrarily small positive real number. 

Take 
$\lambda$ to be the partition $(n,n-1)$ and $\kappa \preceq \nu \preceq \lambda$ to be given by $\kappa=(n-1)$ and $\nu=(n)$. 
 Hence the tableau $\tspec \in \tab (\lambda/\kappa; \nu/\kappa)$ is the following (in which the position of $1$ is given by the single box of $\nu/\kappa$):
 
\ytableausetup{smalltableaux}
\begin{ytableau}
\none & \none & \none &\none &\none & 1
\\
2 & 3 & . & . & \scriptstyle{n}
\end{ytableau}.

A straightforward calculation gives 
$ 
\coeff ((n,n-1),(n),(n-1)) = \frac{1}{n}.
$ 
\end{rem}

%%%%%%%%%%%%%%%%%%%%%%%%%%%%%%%%%%%%%%%%%%%%%%%%%%%%%%%%%%%%%%%%
\subsection{Proof of Proposition \ref{prop:coeff}}
\label{subsect:proof}

 Throughout this section, Hypothesis \ref{hyp:kappa_nu_lambda} is in force.
  
 In order to exploit the explicit action of the symmetric group given by the action of the Coxeter generators, the following elementary result is used:
 
 \begin{lem}
 \label{lem:transpositions}
 For $n \in \nat$ and $1 \leq j <n$, in $\sym_n$ one has
 $
 (1,j+1) = s_j (1,j) s_j.
 $
 \end{lem}
 
 \begin{nota}
 Set $\rho_2= \mathrm{id}$ and, recursively for $j>2$, 
 $ 
 \rho_j = s_{j-1} \rho_{j-1},
 $  
 so that $\rho_j$ is the cycle $(j,j-1,  \ldots, 2)$ given as an explicit product of Coxeter generators.
 \end{nota}
 
 \begin{lem}
 \label{lem:Y_rho}
 \[
 \sum_{i \in \mathbf{b-a +1}} (1, i) Y_{\lambda,\nu, \kappa}
= 
Y_{\lambda, \nu,\kappa}
+ 
\sum_{j=2}^{b-a+1}
\rho_j (1 2) Y_{\lambda,\nu, \kappa}. 
 \]
 \end{lem}
 
\begin{proof}
Lemma \ref{lem:transpositions} implies that, for $j \geq 2$, $(1,j) = \rho_j (1,2) \rho_j^{-1}$. Now $\rho_j^{-1}$ lies in the subgroup $\aut (\{ 2, \ldots , b-a+1 \}) \cong \sym_{b-a} \subset \sym_{b-a+1}\cong \aut (\mathbf{b-a})$, hence acts trivially upon $Y_{\lambda,\nu,\kappa}$, by Lemma \ref{lem:Y_invariance}. The result follows.
\end{proof}

Recall that, for $T \in \tab (\lambda/\kappa)$ and a Coxeter generator $s_i$, one has the two, mutually exclusive possibilities:
 \begin{enumerate}
 \item 
 $i$, $i+1$ are in the same row of $T$ and $s_i \Phi_T = \Phi_T$; 
 \item 
 $s_i T \in \tab (\lambda/\kappa)$ (i.e., $s_i T$ is standard) and $s_i \Phi_T$ is a linear combination of $\Phi_T$ and $\Phi_{s_i T}$.
 \end{enumerate}
This places a strong restriction on the tableaux in $\tab(\lambda/\kappa; \nu/\kappa)$ for which the corresponding term in $Y_{\lambda, \nu,\kappa}$ can contribute to $\coeff (\lambda, \nu,\kappa)$, motivating the following definition:

\begin{defn}
\label{defn:Ybar}
\ 
\begin{enumerate}
\item 
If $a=b$, set $\langle \tspec \rangle =\{ \tspec  \}$.   If $b>a$, let $\langle \tspec \rangle \subset \tab (\lambda/\kappa; \nu/ \kappa)$ be the set of standard tableaux of the following form:
\[
s_2^{\epsilon_2}
s_3^{\epsilon _3}
\ldots 
s_{b-a}^{\epsilon_{b-a}}
\tspec
\]
where $\epsilon_i \in \{0,1\}$ such that $\epsilon_i=0$ if $i$, $i+1$ lie in the same row of the standard tableau $s_{i+1}^{\epsilon _{i+1}}
\ldots 
s_{b-a}^{\epsilon_{b-a}}
\tspec$. 

By convention, $\epsilon_1:= 0$; in particular, if $b-a=1$, then $\langle \tspec \rangle =\{ \tspec  \}$.
\item 
Set $$
\overline{Y}_{\lambda,\nu, \kappa}:= 
\sum_{T \in \langle \tspec \rangle }
\beta_{T|_{\lambda/\nu}} \Phi_T 
\in S^{\lambda/\kappa}.
$$
\end{enumerate}
\end{defn}

\begin{rem}
The `admissibility' condition imposed on the $\epsilon_i$ in the definition of $\langle \tspec \rangle$ ensures that, given a tableau $s_2^{\epsilon_2}
s_3^{\epsilon _3}
\ldots 
s_{b-a}^{\epsilon_{b-a}}
\tspec$ in $\langle \tspec \rangle$, for all $1 \leq k \leq b-a$, the tableau:
\[
s_k^{\epsilon_k}
\ldots 
s_{b-a}^{\epsilon_{b-a}}
\tspec
\]
is standard.
\end{rem}

\begin{lem}
\label{lem:reduce_Ybar}
$\coeff  (\lambda, \nu , \kappa)$ is equal to the coefficient of $\Phi_\tspec$ in
\[
\overline{Y}_{\lambda, \nu,\kappa}
+ 
\sum_{j=2}^{b-a+1}
\rho_j (1 2) \overline{Y}_{\lambda,\nu, \kappa}. 
 \]
 \end{lem}
 
 \begin{proof}
By Lemma \ref{lem:Y_rho},  $\coeff  (\lambda, \nu , \kappa)$ is equal to the coefficient of $\Phi_\tspec$ in 
 $Y_{\lambda, \nu,\kappa}
+ 
\sum_{j=2}^{b-a+1}
\rho_j (1 2) Y_{\lambda,\nu, \kappa}$. 
 
By construction, $\overline{Y}_{\lambda, \nu,\kappa}$ is the sum of the terms appearing in $Y_{\lambda,\nu,\kappa}$ that   can contribute non-trivially to $\coeff  (\lambda, \nu , \kappa)$. 
 \end{proof}
 
The cases of small $b-a$ illustrate the basic behaviour. For $b-a \in \{0,1\}$, one has:

\begin{exam}
\label{exam:b-a<2}
\ 
\begin{enumerate}
\item 
If $b=a$, then $\overline{Y}_{\lambda, \nu,\kappa}=\Phi_\tspec$, where $\tspec$ has a single box, labelled by $1$. Clearly $\coeff (\lambda, \nu, \kappa)=1$ in this case.
\item 
If $b-a=1$, then again one has $\overline{Y}_{\lambda, \nu,\kappa}=\Phi_\tspec$, where now $\tspec$ has two boxes, labelled $1$ and $2$. In this case, 
$ \coeff (\lambda, \nu, \kappa)$ is the coefficient of $\Phi_\tspec $ in $\Phi_\tspec + s_1 \Phi_\tspec$. This is equal to $1 + \frac{1}{r_1 (\tspec)}$. Since $r_1 (\tspec) \in \zed \backslash \{0, 1\}$, this is well-defined and positive.
\end{enumerate}
\end{exam}

The next case exhibits crucial new ingredients:

\begin{exam}
\label{exam:b-a=2}
Suppose that $b-a =2$, thus $\tspec$ has three boxes, labelled $1$, $2$ and $3$,  and $r_2 (\tspec) \geq 1$, with equality if and only if $2$ and $3$ occur in the same row. In this case:
\[
\overline{Y}_{\lambda, \nu,\kappa}=
\Phi_\tspec
+ \beta \Phi_{s_2 \tspec} 
\]
where $\beta = \sqrt \frac{r-1}{r+1}$ for $r= r_2 (\tspec)$. (Here, if $s_2 \tspec$ is not standard, $\Phi_{s_2 \tspec}$ should be understood to be zero; one also has $\beta=0$.)

By definition, $\coeff (\lambda, \nu, \kappa)$ is the coefficient of $\Phi_\tspec$ in 
\[
\overline{Y}_{\lambda, \nu,\kappa}
+
s_1 \overline{Y}_{\lambda, \nu,\kappa}
+
s_2 s_1\overline{Y}_{\lambda, \nu,\kappa}.
\]
Only $\Phi_\tspec$ contributes to the first two expressions, giving $1 + \frac{1}{r_1 (\tspec)}$, as in the previous example. 

For the final term:
\begin{enumerate}
\item 
$s_2 s_1 \Phi_\tspec$ contributes $\frac{1}{r_2(\tspec) r_1 (\tspec)}$;
\item 
suppose that $\beta\neq 0$, then 
$s_2 s_1 (\beta \Phi_{s_2 \tspec})$ contributes 
\[
 \frac{1} {r_1 (s_2 \tspec)}\beta  \sqrt {1 -  \frac{1}{r_2 (s_2 \tspec)^2} }
=
 \frac{1} {r_1 (s_2 \tspec)}
 \sqrt \frac{r_2 (\tspec) -1}{r_2 (\tspec) +1}\sqrt {1 -  \frac{1}{r_2 (s_2 \tspec)^2} }
=
\frac{1} {r_1 (s_2 \tspec)} \Big(\frac{ r_2 (\tspec)  -1 }{r_2 (\tspec) }\Big)
\]
by using the action of the Coxeter generators on the Young orthonormal basis, the fact that $r_2 (s_2 \tspec) = - r_2 (\tspec)$ and that $r= r_2 (\tspec)\geq 1$.
\end{enumerate}
Summing these two contributions gives:
\[
\frac{1}{r_1 (s_2 \tspec) r_2 (\tspec)}
\Big(
\frac{r_1 (s_2\tspec)}{r_1 (\tspec)} + r_2 (\tspec) -1 
\Big) 
= \frac{1}{r_1 (s_2 \tspec)}\Big( 1  + \frac{1} {r_1 (\tspec)} \Big), 
\]
where the equality is obtained by using the identity (see Lemma \ref{lem:axial_additivity}):
\[
a(3,1) (s_2 \tspec) = r_2 (s_2 \tspec) + r_1 (s_2 \tspec), 
\]
which, since $a(3,1)(s_2 \tspec) = r_1 (\tspec)$, gives 
$
r_1 (s_2 \tspec) = r_1 (\tspec) + r_2 (\tspec).
$

Putting these facts together, one gets:
\[
\coeff (\lambda, \nu, \kappa) = 
\left\{
\begin{array}{ll}
\Big( 1  + \frac{1} {r_1 (\tspec)} \Big) \Big(1 + \frac{1}{r_1 (s_2 \tspec)} \Big)
&
\mbox{$s_2 \tspec$ standard}
\\
 1  + 2 \frac{1} {r_1 (\tspec)}
& \mbox{otherwise.}
\end{array}
\right.
\] 

A straightforward verification (using that $r_1 (\tspec) \neq 1$) shows that these expressions are equal if one uses the identity $r_1 (s_2 \tspec) = r_1 (\tspec) + r_2 (\tspec)$ to {\em define} $r_1 (s_2 \tspec)$ when $s_2 \tspec $ is not standard, since $r_2(\tspec)=1$ in this case. Equivalently, one can replace $r_1 (s_2 \tspec) $ by $a(3,1) (\tspec)$, giving the unified expression:
\[
\coeff (\lambda, \nu, \kappa) = 
\Big( 1  + \frac{1} {r_1 (\tspec)} \Big) \Big(1 + \frac{1}{a(3,1)(\tspec)} \Big).
\]

The factor $(1 + \frac{1}{a(3,1)(\tspec)})$ should be interpreted as being the value of $\coeff$ obtained when replacing $\lambda/\kappa$ by the skew-diagram given by omitting the box labelled by $2$ and reindexing. This fits into a general inductive scheme, as below.
\end{exam} 

We next give an explicit expression for $\coeff (\lambda, \nu, \kappa)$ (see Proposition \ref{prop:calculate_coeff}), using the ingredients used in Example \ref{exam:b-a=2}.

\begin{nota}
\label{nota:coeff(T)}
For $T$ in $\langle \tspec \rangle$, let $\coeff  (\lambda,\nu,\kappa)(T)$ be the coefficient of $\Phi_\tspec$ in 
$$
\beta_{T|_{\lambda/\nu}} \big( 
\Phi_T + 
\sum_{j=2}^{b-a+1}
\rho_j (1 2) \Phi_T
\big).
$$
\end{nota}

Lemma \ref{lem:reduce_Ybar} implies:

\begin{lem}
\label{lem:sum_T}
$
\coeff  (\lambda,\nu,\kappa)=
\sum_{T \in \langle \tspec \rangle}
\coeff  (\lambda,\nu,\kappa)(T).
$
\end{lem}

To state Proposition \ref{prop:calculate_coeff}, we require the following notation:

\begin{nota}
Suppose that $b-a \geq 1$ and  $T\in \langle \tspec \rangle$ given by the sequence $(\epsilon_j)|_{2\leq j \leq  b-a}$.
\begin{enumerate}
\item 
for $1 \leq i \leq b-a$, denote by $\re_i(T)$  the axial distance from $i+1$ to $i$ in $s_{i+1}^{\epsilon _{i+1}}
\ldots 
s_{b-a}^{\epsilon_{b-a}}
\tspec$ (thus $\re_1 (T) = r_1 (T)$); 
\item 
denote by $J(T):= \inf \{ i \ | \  \epsilon_j =0 \ \forall j > i \}$.
\end{enumerate}
\end{nota}

We note the following important property of 
the axial distances $\re_i(T)$:

\begin{lem}
\label{lem:positivity}
For $T \in \langle \tspec \rangle$ given by the sequence $(\epsilon_j)|_{2\leq j \leq  b-a}$ and $2 \leq j \leq b-a$:
\[
\re_j(T) \geq 1 
\]
with equality if and only if $j$, $j+1$ lie in the same row of $s_{j+1}^{\epsilon _{j+1}}
\ldots 
s_{b-a}^{\epsilon_{b-a}}
\tspec$.

If  $\tspec = T\rev$ (i.e., the box labelled by $1$ is the leftmost box of the skew tableau $\tspec$), then $\re_1(T)\geq 1$.
\end{lem}

\begin{proof}
The first statement is proved by increasing induction upon $b-a$ by using Lemma \ref{lem:Trev}.

The box labelled $1$ does not intervene, hence (up to relabelling) one may restrict to ${\lambda/\nu}$. Restricted to $\lambda /\nu$, by construction  one has $\tspec \equiv T\rev$. In particular, this gives $\re_{b-a} (T)= r_{b-a}(\tspec) \geq 1$ by Lemma \ref{lem:Trev}. The statement on the equality is clear. 

For the inductive step, consider $s_{b-a}^{\epsilon_{b-a}} \tspec$. Forgetting the box labelled by $b-a+1$ (this is the maximal label, after the relabelling), this gives a sub skew diagram $\lambda' /\kappa$ and the corresponding restricted tableau is $\tspec \in \tab(\lambda'/\kappa; \nu/\kappa)$. Hence the inductive hypothesis applies.

Finally, if $\tspec = T\rev$, it is clear that $\re_1 (T)\geq 1$ also. 
\end{proof}

\begin{prop}
\label{prop:calculate_coeff}
Suppose that $b-a \geq 1$. For $T \in \langle \tspec \rangle$ given by the sequence $(\epsilon_j)|_{2\leq j \leq  b-a}$,
\[
\coeff  (\lambda,\nu,\kappa)(T)
=
\delta_{T,\tspec}
+  
\sum_{k=J(T)}^{b-a} 
\prod_{1\leq j \leq  k}
\Big(\frac{1}{\re_j(T)}\Big)^{1-\epsilon_j}
\Big(\frac{\re_j(T)-1}{\re_j(T)}\Big)^{\epsilon_j},
\]
where $\delta_{T,\tspec}$ is $1$ if $T = \tspec$ and $0$ otherwise.
\end{prop}

\begin{proof}
Firstly, Proposition \ref{prop:invt_skew} allows the calculation of $\beta_{T|_{\lambda/\nu}}$ by a straightforward induction on $b-a$. Namely, one shows that 
\[
\beta_{T|_{\lambda/\nu}}
= 
\prod_{j=2}^{b-a}
\Big(
\sqrt{\frac{\re_j-1 }{\re_j+1}}
\Big)^{\epsilon_j}
\]
(using $0^0=1$ in the case $\re_j=1$, for which $\epsilon_j=0$).

Now, consider $\Phi_T + 
\sum_{j=2}^{b-a}
\rho_j (1 2) \Phi_T$. Clearly $\Phi_T$ has a non-trivial coefficient of $\Phi_\tspec$ if and only if $T= \tspec$, which is equivalent to $\epsilon_j =0$ for $1 \leq j \leq b-a$. 
In this case, $\beta_{T|_{\lambda/\nu}}=1$, so this accounts for the term $\delta_{T,\tspec}$ in the statement.

Now consider $\rho_k (1 2) \Phi_T$ for $2 \leq k \leq b-a+1$. A straightforward calculation using the action of the Coxeter generators on the Young orthonormal basis shows that the coefficient of $\Phi_\tspec$ is non-zero if and only if 
$\epsilon_i =0$ for all $i>k$. This condition is equivalent to $k \geq J(T)$, by definition of the latter. 

Moreover, in this case, arguing as in Example \ref{exam:b-a=2}, the coefficient is equal to 
$$
\prod_{1\leq j \leq  k}
\Big(\frac{1}{\re_j(T)}\Big)^{1-\epsilon_j}
\Big(\sqrt{\frac{\re_j(T)^2-1}{\re_j(T)^2}}\Big)^{\epsilon_j}.
$$

Again as in the Example, using the identity
$$
\sqrt{\frac{\re_j(T)-1 }{\re_j(T)+1}} \sqrt{\frac{\re_j(T)^2 -1}{\re_j(T)^2}} = \frac{\re_j(T) -1}{\re_j(T)},
$$
under the above hypothesis that $\epsilon_i =0$ for all $i>k$,
one has the equality 
$$
\beta_{T|_{\lambda/\nu}}
\prod_{1\leq j \leq  k}
\Big(\frac{1}{\re_j(T)}\Big)^{1-\epsilon_j}
\Big(\sqrt{\frac{\re_j(T)^2-1}{\re_j(T)^2}}\Big)^{\epsilon_j}
=
\prod_{1\leq j \leq  k}
\Big(\frac{1}{\re_j(T)}\Big)^{1-\epsilon_j}
\Big(\frac{\re_j(T)-1}{\re_j(T)}\Big)^{\epsilon_j}.
$$

Summing the terms over $k\geq J(T)$ gives the required result.
\end{proof}

\begin{rem}
\ 
\begin{enumerate}
\item 
The expression in the statement of Proposition \ref{prop:calculate_coeff} is defined over $\mathbb{Q}$.
\item
If $b-a\geq 1$, it  can be rewritten as 
\[
\coeff  (\lambda,\nu,\kappa)(T)
=
\delta_{T,\tspec}
+  
\frac{1}{\re_1(T)}
\sum_{k=J(T)}^{b-a} 
\prod_{2\leq j \leq  k}
\Big(\frac{1}{\re_j(T)}\Big)^{1-\epsilon_j}
\Big(\frac{\re_j(T)-1}{\re_j(T)}\Big)^{\epsilon_j}.
\]
\end{enumerate}
\end{rem}

Proposition \ref{prop:calculate_coeff} leads to the easy case of Proposition \ref{prop:coeff}:

\begin{prop}
\label{prop:easy}
Suppose that $\lambda/\nu$ is non-empty and that  $\tspec = T\rev$. Then $\coeff  (\lambda, \nu ,\kappa) >1$.
\end{prop}

\begin{proof}
The hypothesis together with Lemma \ref{lem:positivity}  ensure that $\re_i(T)>0$ for all $1 \leq i \leq b-a$. Hence, by Proposition \ref{prop:calculate_coeff}, all contributions to $\coeff  (\lambda, \nu, \kappa)$ are non-negative and the contribution from $\tspec$, $1+ \frac{1}{\re_1 (T)}$, is strictly greater than $1$.
\end{proof}

It remains to treat the cases not covered by Proposition \ref{prop:easy}, namely when $\lambda/\nu$ is non-empty and  $\tspec \neq T\rev$. The difficulty in this case stems from the fact that, for $T \in \langle \tspec \rangle$,  $r_1 ^\epsilon (T)$ need not be positive. 

Under the above hypothesis,  the left-most box in $\tspec$ is labelled by $2$, with the box labelled $1$ occurring as a (different) inner corner. Then, by `removing' the box  labelled $2$ from the skew tableau and relabelling the remaining boxes of the skew tableau (other than $1$) by $n\mapsto n-1$, one obtains $\tspec$ associated to the sub skew diagram without the left-most box. This is the basis of an inductive strategy for understanding $\coeff  (\lambda, \nu, \kappa)$.

\begin{nota}
\label{nota:plus}
Suppose that $\lambda/\nu$ is non-empty and that  $\tspec \neq T\rev$. Let $\nu^+$ (respectively $\kappa^+$) be the partition obtained from $\nu$ (resp. $\kappa$) by increasing the length of the first {\em column} by one. (At the level of the Young tableaux of $\nu$ and $\kappa$, this corresponds to adding the box labelled $2$ to the respective Young tableaux; respectively, for the skew tableaux $\lambda/\nu$ and $\lambda /\kappa$, it corresponds to removing it.) 
\end{nota}

\begin{rem}
The hypothesis upon $\tspec$ serves here only to eliminate the possibility that $\tspec$ contains the configuration $\begin{ytableau} 1 & 2 & 3\end{ytableau}$, in which case it is not possible to `remove' the box labelled by $2$ and retain a skew diagram. 
\end{rem}

By construction, one has the following:

\begin{lem}
\label{lem:remove2}
Suppose that $\lambda/\nu$ is non-empty and that  $\tspec \neq T\rev$, then  $\nu^+ \vdash a+1$, $\kappa^+ \vdash a$ and 
\[
\hs{\lambda} \preceq \kappa^+ \preceq \nu^+ \preceq \lambda. 
\]
Moreover, the skew diagram $\lambda / \kappa^+$ is obtained from that of $\lambda /\kappa$ by removing the left-most box.

There is a bijection between $\tab (\lambda/ \kappa^+;\nu^+/\kappa^+)$ and the subset of $\tab (\lambda /\kappa; \nu/\kappa)$ of tableaux such that the left-most box is labelled by $2$.
\end{lem}

\begin{nota}
Suppose that $\lambda/\nu$ is non-empty and that  $\tspec \neq T\rev$. 
For $T \in \tab (\lambda/\kappa^+; \nu^+/\kappa^+)$, denote by $T^{[2]} \in \tab (\lambda/ \kappa; \nu /\kappa)$ the tableau obtained by the bijection of Lemma \ref{lem:remove2}.

Denote by $\overline{\tspec} \in  \tab (\lambda/\kappa^+; \nu^+/\kappa^+)$ the tableau associated to the triple $(\lambda, \nu^+, \kappa^+)$ as in Notation \ref{nota:skew}. (This is obtained from $\tspec$ by forgetting the box labelled $2$ and renumbering.)
\end{nota}

The definition of $\langle \tspec \rangle$ given in Definition \ref{defn:Ybar} leads to the following, which is the basis for an inductive approach.

\begin{lem}
\label{lem:decompose_langle_tspec_rangle}
Suppose that $|\lambda/\nu| \geq 2$ and that  $\tspec \neq T\rev$. 
The set $\langle \tspec \rangle$ decomposes as 
$ 
\langle \tspec \rangle = 
\langle \tspec \rangle_0
\amalg
\langle \tspec \rangle_1
 $, 
where $\langle \tspec \rangle_0 = \{ T^{[2]} | T \in \langle \overline{\tspec} \rangle \}$ and 
$\langle \tspec \rangle_1 = \{ s_2 (T^{[2]}) | T \in \langle \overline{\tspec} \rangle \mbox{ such that $s_2 (T^{[2]})$ is standard} \}$.
\end{lem}

\begin{proof}
By Definition \ref{defn:Ybar}, an element of $\langle \tspec \rangle$ is determined by an expression
$$
s_2^{\epsilon_2}
s_3^{\epsilon _3}
\ldots 
s_{b-a}^{\epsilon_{b-a}}
\tspec
$$
for a unique sequence $(\epsilon_i)$. The decomposition into the two components corresponds to the two possibilities $\epsilon_2 \in \{ 0, 1\}$.
\end{proof}

\begin{prop}
\label{prop:inductive_theta}
Suppose that $|\lambda/\nu| \geq 2$ with  $\tspec \neq T\rev$. Then 
\[
\coeff  (\lambda, \nu, \kappa)
=
\Big(1 + \frac{1}{r_1 (\tspec)}\Big) \coeff  (\lambda, \nu^+, \kappa^+).
\]
\end{prop}

\begin{rem}
This result generalizes the case $|\lambda/\nu| = 2$ that is treated in Example \ref{exam:b-a=2}. In particular, in the example, it was shown that 
 \[
\coeff  (\lambda, \nu, \kappa)
=
\Big(1 + \frac{1}{r_1 (\tspec)}\Big) \Big( 1 + \frac{1} {a(3,1)(\tspec)} \Big) .
\]
To conclude in this case, it suffices to observe that $\coeff  (\lambda, \nu^+, \kappa^+)    = 
\Big( 1 + \frac{1} {a(3,1)(\tspec)} \Big)$, which follows from Example \ref{exam:b-a<2}, using the fact that $r_1 (\overline{\tspec}) = a (3,1) (\tspec)$.
\end{rem}

\begin{proof}[Proof of Proposition \ref{prop:inductive_theta}]
For the purposes of this proof, we introduce the following notation: 
\begin{enumerate}
\item 
for $T' \in \langle \tspec \rangle$, define $\coeff_{\geq 2} (T')$ by 
\[
\coeff (\lambda, \nu, \kappa) (T') = \delta_{T', \tspec} \big( 1 + \frac{1} {r_1 (T')} \big) + \coeff_{\geq 2} (T'); 
\]
\item 
for $T \in \langle \overline{\tspec} \rangle$, define $\coeff_{\geq 1}(T)$ and $\overline{\coeff}_{\geq 1} (T)$ by
\[
\coeff (\lambda, \nu^+, \kappa^+) (T) =
\delta_{T, \overline{\tspec}}  + \coeff_{\geq 1} (T)
=
 \delta_{T, \overline{\tspec}}  + \frac{1}{r_1 (T)} \overline{\coeff}_{\geq 1} (T). 
\]
\end{enumerate}
The key to the proof is to relate $\coeff_{\geq 2} (T^{[2]})$ and $\coeff_{\geq 2} (s_2 T^{[2]})$ (when $s_2 T^{[2]}$ is standard) to $\coeff_{\geq 1}(T)$ and $\overline{\coeff}_{\geq 1} (T)$. 

First one checks (for example, using the explicit description given in Proposition \ref{prop:calculate_coeff}) that 
\begin{eqnarray*}
\coeff_{\geq 2}(T^{[2]}) &=&
\frac{1}{r_1 (T^{[2]}) r_2(T^{[2]}) }\overline{\coeff}_{\geq 1} (T) 
\\
&=&
\frac{r_1 (T)}{r_1 (T^{[2]}) r_2(T^{[2]}) }\coeff_{\geq 1} (T), 
\end{eqnarray*}
where the second equality follows from the relationship between $\coeff_{\geq 1}(T)$ and $\overline{\coeff}_{\geq 1} (T)$. 

Now, $s_2 T^{[2]}$ is {\em not} standard if and only if the configuration $\begin{ytableau}2 & 3 \end{ytableau}$ occurs in $T^{[2]}$. In this case the additivity of axial distances given by Lemma \ref{lem:axial_additivity} implies the equality 
\[
r_1 (T) = 1 + r_1(T^{[2]}),
\]
using $r_2 (T^{[2]}) =1$. Thus, in this case, one has the equality:
\[
\coeff_{\geq 2}(T^{[2]}) 
= 
\Big( 1 + \frac{1}{r_1 (\tspec) } \Big) \coeff_{\geq 1}(T),
\]
using that $r_1 (T^{[2]}) = r_1 (\tspec)$, by construction of $T^{[2]}$.

Now consider the case where $s_2 T^{[2]}$ is standard. One has the identity $r_1 (T)= r_1 (s_2 T^{[2]})$ and, reasoning as for $\coeff_{\geq 2} (T^{[2]})$, one deduces the equality:
\[
\coeff_{\geq 2} (s_2 T^{[2]}) = \frac{r_2 (T^{[2]}) -1 }{r_2 (T^{[2]}) } \coeff_{\geq 1} (T).
\]
(compare also Example  \ref{exam:b-a=2}). 

Thus, 
\[
\coeff_{\geq 2}(T^{[2]}) 
+ 
\coeff_{\geq 2} (s_2 T^{[2]})
= 
\frac{1}{r_2 (T^{[2]}) } 
\Big( 
\frac{r_1(T)}{r_1(T^{[2]}) } 
+ (r_2 (T^{[2]}) -1 ) 
\Big) 
\coeff_{\geq 1} (T) .
\]
Again using Lemma \ref{lem:axial_additivity}, one has $r_1 (T) = r_1 (T^{[2]}) + r_2 (T^{[2]})$, so that 
\[
\frac{r_1(T)}{r_1(T^{[2]}) } 
+ (r_2 (T^{[2]}) -1 )
= 
r_2(T^{[2]}) \Big ( 1 + \frac{1}{r_1 (T^{[2]})} \Big).
\]
This gives:
\[
\coeff_{\geq 2}(T^{[2]}) 
+ 
\coeff_{\geq 2} (s_2 T^{[2]})
= 
\Big ( 1 + \frac{1}{r_1 (\tspec)} \Big)\coeff_{\geq 1} (T) 
 .
\]

Summing over $T \in \langle \overline{\tspec} \rangle$ gives the result, by Lemma \ref{lem:decompose_langle_tspec_rangle} together with Lemma \ref{lem:sum_T}.
\end{proof}

\begin{proof}[Proof of Proposition \ref{prop:coeff}]
By inspection, the result holds in the case $|\lambda/\nu|\leq 1$ (for example, use 
Example \ref{exam:b-a<2}). Proposition  \ref{prop:easy} establishes the result in the case  $\tspec = T\rev$. These form the initial cases for an inductive proof. 

The inductive step treats the case $|\lambda/\nu| \geq 2$ and $\tspec \neq T\rev$. Proposition \ref{prop:inductive_theta} gives the equality:
\[
\coeff  (\lambda, \nu, \kappa)
=
\Big(1 + \frac{1}{r_1 (\tspec)}\Big) \coeff  (\lambda, \nu^+, \kappa^+).
\]

The inductive hypothesis implies that $\coeff  (\lambda, \nu^+, \kappa^+)>0$ and $(1 + \frac{1}{r_1 (\tspec)})\geq \frac{1}{2}$, since $\frac{1}{r_1 (\tspec)} \geq - \frac{1}{2}$. The result follows.
\end{proof}

 \section{Higher $\finj$-homology of $\kring \hfi (- , \mathbf{b})\trans$}
 \label{sect:homology}
 
 The purpose of this Section is to place Theorem \ref{thm:coker}  in the more general context of studying $\finj$-homology. Recall that $\fb \subset \finj$ is the maximal subgroupoid; the category of functors from $\finj$ to $\kring$-modules is denoted $\fimod$ (respectively $\fbmod$ for functors on $\fb$) and sometimes referred to as $\finj$-modules (resp. $\fb$-modules).
 
There is an exact extension functor 
 $
 \fbmod \rightarrow \fimod
 $ 
 that sends a $\fb$-module $G$ to the $\finj$-module taking the same values and on which 
 morphisms of $\finj$ that are not bijections act via zero.  This has a left adjoint 
 \[
 H^\finj_0 : \fimod \rightarrow \fbmod.
 \] 
 This identifies explicitly as follows: for $F$ a $\finj$-module, $(H^\finj_0 F)(\mathbf{a})$ is the cokernel of the map
 \[
 F(\mathbf{a-1}) \uparrow _{\sym_{a-1}}^{\sym_a} 
 \rightarrow 
 F(\mathbf{a}) 
 \] 
 induced by $\iota_{a-1, a}$. Using this, one has:

 \begin{prop}
 \label{prop:coker_H0}
 For $a, b \in \nat$, there is a natural isomorphism 
 \[
 \mathrm{Coker} \tr_{a,b} \cong \big(H^\finj_0 \kring \hfi (-, \mathbf{b})\trans \big) (\mathbf{a}). 
 \]
 \end{prop}
 
 By definition, $\finj$-homology is given by the left derived functors of $H^\finj_0$. It is natural to seek to identify the functors
 \[
 (\mathbf{a}, \mathbf{b}) 
 \mapsto 
 \big(H^\finj_n \kring \hfi (-, \mathbf{b})\trans \big) (\mathbf{a})
 \]
 for $n \in \nat$, 
 generalizing the case $n=0$ given by Theorem \ref{thm:coker}.
 
 The main result of this Section, Proposition \ref{prop:subset_H_finj}, gives a lower bound for this $\finj$-homology which, conjecturally, coincides with the $\finj$-homology.
 
 %%%%%%%%%%%%%%%%%%%%%%%%%%%%%%%%%%%%%%%%%%%%%%%%%%%%%%%%%%%%%%%%%%
\subsection{$\finj$-homology - some recollections}

The $\finj$-homology of $F \in \ob \fimod$ can be calculated as the homology of an explicit Koszul complex $\kz_\bullet F$ in $\fbmod$ such that $\kz_\bullet$ is an exact functor from $\fimod$ to chain complexes in $\fbmod$ \cite{MR3654111,MR3603074}.

To describe the part of the structure that is required below, first recall the convolution product on $\fbmod$. This is the  symmetric monoidal structure $(\fbmod , \conv, \kring \langle 0 \rangle)$, given for $\fb$-modules $G_1$, $G_2$ by 
\[
G_1 \conv G_2 (X) = \bigoplus_{X = X_1 \amalg X_2} G_1(X_1) \otimes G_2(X_2),
\]
where the sum is taken over the set of ordered decompositions of $X$ into two subsets. (The $\fb$-module $\kring \langle 0 \rangle$ is $\kring$ evaluated on $\emptyset$ and zero evaluated on a non-empty finite set.)

Next, recall the {\em orientation} $\fb$-module $\ori$ that is given on $X \in \ob \fb$ by
\[
\ori (X) = \Lambda^{|X|} (\kring X),
\]
the top exterior power of the $\kring$-linearization of $X$. Thus $\ori (\mathbf{n})$ is the signature representation of $\sym_n$, for $n \in \nat$. This is understood to have  {\em homological degree} $n$. 

For $F$ a $\finj$-module, by restricting to $\fb$ one can form the convolution product $\ori \conv F$ in $\fbmod$. The Koszul complex of $F$ in $\fbmod$ has the form
\[
(\kz_\bullet F, d)
= ( \ori \conv F, d),
\]
in which the differential takes into account the $\finj$-module structure of $F$.

\begin{rem}
In Section \ref{subsect:revisit_kz}, the associated complex given by passage to Schur bifunctors is described explicitly in the case of interest, i.e., when $F$ is $\kring \hfi (-, \mathbf{b}) \trans$. 
\end{rem}

Evaluating on $\mathbf{a}$, for $a \in \nat$, in homological degree $n$ with $0 \leq n \leq a$, one has:
\[
\kz_n F (\mathbf{a}) 
\cong 
\big(\mathrm{sgn}_n \otimes F(\mathbf{a-n})\big) \uparrow _{\sym_n \times \sym_{a-n} }^{\sym_a} 
\]
as $\sym_a$-modules. For $n >a$, $\kz_n F (\mathbf{a}) =0$.

%%%%%%%%%%%%%%%%%%%%%%%%%%%%%%%%%%%%%%%%%%%%%%%%%%%%%%%%%%%%%%%%%%%%%%%%%%%%%%%
\subsection{The case $F=\kring \hfi (- , \mathbf{b})\trans$}

Henceforth $\kring$ is taken to be a field of characteristic zero. 

We fix $b \in \nat$ and consider the $\finj$-module $F=\kring \hfi (- , \mathbf{b})\trans$. Proposition \ref{prop:inj_comp_fact} identified the composition factors of the underlying $\fb$-module in $\sym_b$-modules. Thus we can identify the composition factors occurring in the Koszul complex $\kz_\bullet F$, as below.

The following is the counterpart of Definition \ref{defn:horizontal_strip}:

\begin{defn}
\label{defn:vstrip}
For partitions $\lambda$, $\mu$ such that $\mu \preceq \lambda$, the skew partition $\lambda/ \mu$ is a vertical strip if it contains at most one box in each row. 
\end{defn}

The following gives the counterpart of $\hs{\lambda}$ of Notation \ref{nota:hs}:

\begin{nota}
\label{nota:vstrip}
For a partition $\lambda$, denote by $\vstrip{\lambda}$ the partition  $\vstrip{\lambda}_i = \lambda_i -1$ if $\lambda_i>0$. (In terms of Young diagrams,  $\vstrip{\lambda}$ is obtained from $\lambda$ by removing the first column.)
\end{nota}

\begin{rem}
\label{rem:vstrip}
Recall that the transpose of a partition $\lambda$ is the partition $\lambda^\dagger$ what has Young diagram given by reflecting that of $\lambda$ in the diagonal, thus interchanging rows and columns. For instance, the transpose of $(n)$ is $(1^n)$ and vice versa.

If $\mu \preceq \lambda$ then $\mu^\dagger \preceq \lambda^\dagger$ and  $\lambda/\mu$ is a vertical strip if and only if $\lambda^\dagger /\mu^\dagger$ is a horizontal strip. In particular, the `transpose' to Lemma \ref{lem:hs} gives that $\lambda/ \mu$ is a vertical strip if and only if $\vstrip{\lambda} \preceq \mu \preceq \lambda$.
\end{rem}

Pieri's rule gives:
 
\begin{lem}
\label{lem:pieri}
For $\nu \vdash t $ and $n \in \nat$, 
\[
(S^\nu \boxtimes \mathrm{sgn}_n)\uparrow_{\sym_t \times \sym_n}^{\sym_{t+n}}
\cong 
\bigoplus_{ 
\substack{\mu \vdash t+n \\ \vstrip{\mu} \preceq \nu \preceq \mu}}
S^\mu.
\]
Equivalently, the sum is indexed over partitions $\mu$ such that $\nu \preceq \mu$ and $\mu /\nu$ is a vertical strip.
\end{lem}

\begin{hyp}
\label{hyp:ab}
Henceforth the bidegree $(a,b)\in \nat^{\times 2}$ is fixed. (Note that we do not suppose that $a \leq b$.) 
\end{hyp}

Consider the $n$th homological degree term of the Koszul complex:
\[
\big(\mathrm{sgn}_n \otimes \kring \hfi (\mathbf{a-n}, \mathbf{b})\big) \uparrow_{\sym_b \times (\sym_n \times \sym_{a-n}) } ^{\sym_b \times \sym_a}.
\]
This is zero if either $a >b+n $ or $a <n$. 

\begin{nota}
\label{nota:pair}
For $(a,b)$ and $n$ as above and $\nu \vdash a-n$, denote by $\pair(a, b;\nu)$ the set of pairs of partitions $(\lambda \vdash b, \mu\vdash a)$ such that the following conditions are satisfied:
\begin{enumerate}
\item 
$\hs{\lambda} \preceq \nu \preceq \lambda$;
\item
$\vstrip{\mu} \preceq \nu \preceq \mu$.
\end{enumerate}
(Thus $\lambda/\nu$ is a horizontal strip and $\mu/\nu$ is a vertical strip.)
\end{nota}

Proposition \ref{prop:inj_comp_fact}, Lemma \ref{lem:pieri}, together with the definition of $ \pair (a,b;\nu)$ imply the following:

\begin{lem}
\label{lem:kz_n}
For $(a,b)$ and $n$ as above, there is an isomorphism of $\sym_b \times \sym_a$-modules:
\[
\big( \kz_n \kring \hfi (-, \mathbf{b}) \big)(\mathbf{a})
\cong 
\bigoplus_{\nu \vdash a-n}\bigoplus_{(\lambda,\mu)\in \pair (a,b;\nu)} S^\lambda  \boxtimes S^\mu.
\]
\end{lem}

\begin{rem}
\label{rem:lambda_nu_kz}
Given $(\lambda,\mu)\in \pair (a,b;\nu)$, the homological degree $n$ can be recovered as $|\mu/\nu|$.
\end{rem}

%%%%%%%%%%%%%%%%%%%%%%%%%%%%%%%%%%%%%%%%%%%%%%%%%%%%%%
\subsection{Reindexing by $\M$}

So as to focus upon the simple composition factors $S^\lambda \boxtimes S^\mu$, we  introduce the following:

\begin{nota}
\label{nota:M}
For a fixed pair $(\lambda \vdash b, \mu \vdash a)$ let $\M (\lambda, \mu)$ denote the set of partitions $\nu$ such that both the following conditions hold:
\begin{enumerate}
\item 
$\hs{\lambda} \preceq \nu \preceq \lambda$;
\item
$\vstrip{\mu} \preceq \nu \preceq \mu$.
\end{enumerate}
\end{nota}

\begin{nota}
For partitions $\lambda$, $\mu$, let $\lambda \cap \mu$ and $\lambda \cup \mu$ be respectively the infimum and supremum with respect to $\preceq$. (With respect to the associated Young diagrams, these correspond to the intersection and the union, in the obvious sense).
\end{nota}

By construction, Lemma \ref{lem:kz_n} can be reformulated as:

\begin{lem}
\label{lem:M_kz}
For $(a,b)$ and $n$ as above, there is an isomorphism of $\sym_b \times \sym_a$-modules:
\[
\big( \kz_n \kring \hfi (-, \mathbf{b}) \big)(\mathbf{a})
\cong 
\bigoplus_{\lambda \vdash b, \mu \vdash a}
\bigoplus_{\substack{\nu \in \M (\lambda, \mu) \\ \nu \vdash a-n}} S^\lambda  \boxtimes S^\mu.
\]

In particular, $S^\lambda \boxtimes S^\mu $ occurs in the Koszul complex if and only if the following equivalent conditions are satisfied:
\begin{enumerate}
\item 
 $\lambda \cap \mu \in \M (\lambda, \mu)$;
 \item 
 $\M (\lambda, \mu)\neq \emptyset$.
\end{enumerate}
Moreover, $|\M (\lambda, \mu)|$ is the total multiplicity of $S^\lambda \boxtimes S^\mu$  in the Koszul complex (i.e., allowing arbitrary homological degree). 
\end{lem}

Clearly one has:

\begin{lem}
\label{lem:chain}
\ 
\begin{enumerate}
\item 
Suppose that $\nu \in \M (\lambda, \mu)$, then $\nu \preceq (\lambda \cap \mu)$. 
\item 
Moreover, if $\nu \preceq \gamma \preceq (\lambda \cap \mu)$ and $\nu \in \M (\lambda, \mu)$, then $\gamma \in \M (\lambda, \mu)$. 
\end{enumerate}
\end{lem}

Recall that, if $\nu' \preceq \nu$ with $|\nu'| = |\nu|-1$, then $\nu'$ is obtained from $\nu$ be removing an outer corner of the Young diagram representing $\nu$. (An outer corner of a Young diagram is a box without neighbours either to the right or below.)
 Then, if $\nu \in \M (\lambda, \mu)$, it is natural to ask under what condition $\nu'$  lies in $\M (\lambda , \mu)$. 
 
 Given an outer corner, by hypothesis this belongs both to the diagram of  $\lambda$ and that of $\mu$. Consider the boxes (potentially) to the right and below this outer corner:
 
 
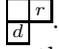
\begin{ytableau}
\  & r  \\
d 
\end{ytableau}. 

\noindent
Here the unlabelled box is an outer corner of $\nu$, by hypothesis,  so the boxes $\begin{ytableau} r \end{ytableau}$ and $\begin{ytableau} d \end{ytableau}$ do not belong to $\nu$; they may possibly  belong  to either the diagram $\lambda$ or to that of $\mu$, but this is not necessarily the case.

\begin{lem}
\label{lem:rd}
In the above situation, $\nu' \in \M (\lambda, \mu)$ if and only if both the following conditions are satisfied:
\begin{enumerate}
\item 
the box $\begin{ytableau} r \end{ytableau}$ does not belong to $\mu/\nu$; 
\item 
the box $\begin{ytableau} d \end{ytableau}$ does not belong to $\lambda/\nu$. 
\end{enumerate}
\end{lem}

\begin{proof}
One checks that $\mu/ \nu'$ is a vertical strip if and only if $\begin{ytableau} r \end{ytableau}$ does not belong to $\mu/\nu$. The `transpose' gives the second condition.
\end{proof}

One also has:

\begin{lem}
\label{lem:outer}
If $\nu \in \M (\lambda, \mu)$, then $(\lambda\cap \mu)/\nu$ is a skew diagram consisting of outer corners of $\lambda \cap \mu$.  
\end{lem}

\begin{proof}
The condition on the skew diagram is equivalent to the following:

\begin{itemize}
\item
the skew diagram $(\lambda\cap \mu)/\nu$ does not contain a sub diagram of the form 
\begin{ytableau}
\   \\
\ 
\end{ytableau}
or 
\begin{ytableau}
\   &
\ 
\end{ytableau}.
\end{itemize}

Suppose that $(\lambda\cap \mu)/\nu$ contains $\begin{ytableau}
\   &
\ 
\end{ytableau}$, then $\mu / \nu$ does also, in particular it is not a vertical strip, 
contradicting the hypothesis $\nu \in \M (\lambda, \mu)$.

The other case is treated by the `transpose' argument.  
\end{proof}

Putting these points together, one arrives at the following concrete description of the set $\M(\lambda, \mu)$: 

\begin{prop}
\label{prop:M_lambda_mu}
Suppose that $(\lambda \cap \mu) \in \M (\lambda, \mu)$ (equivalently, that $\M (\lambda, \mu) \neq \emptyset$) and let $\nu(\lambda, \mu)\preceq (\lambda \cap \mu)$ be the diagram obtained by removing all the outer corners of $(\lambda \cap \mu)$ that satisfy the criterion of Lemma \ref{lem:rd} (taking $\nu = (\lambda \cap \mu)$).

Then $\M (\lambda , \mu ) = \{ \nu'' | \nu (\lambda ,\mu) \preceq \nu'' \preceq (\lambda \cap \mu) \}$. In particular, this has cardinal $2^{|(\lambda \cap \mu)/ \nu(\lambda, \mu)|}$. 
\end{prop}

\begin{proof}
Most of the statement follows directly from Lemma \ref{lem:chain} and Lemma \ref{lem:outer}. The only point that remains to be established is that $\nu (\lambda, \mu)$ belongs to $\M(\lambda,\mu)$. This is  established by using the criterion of Lemma \ref{lem:rd} and the observation that removing outer corners of $\lambda \cap \mu$ does not affect the right and below neighbours of the remaining outer corners.
\end{proof}

%%%%%%%%%%%%%%%%%%%%%%%%%%%%%%%%%%%%%%%%%%%%%%%%%%%%%%%%%%%%%%%%%
\subsection{Criticality}

One distinguishes the critical cases where $\M (\lambda, \mu)= \{ \lambda \cap \mu \}$.
By Lemma \ref{lem:M_kz}  (see also Proposition \ref{prop:M_lambda_mu}), for such cases  $S^\lambda \boxtimes S^\mu$ occurs with multiplicity one in the Koszul complex, in homological degree $|\mu/(\lambda \cap \mu)|$.

\begin{nota}
\label{nota:crit}
Let $\crit (a,b)$ denote the set of pairs of partitions $(\lambda \vdash b, \mu \vdash a)$ such that $\M (\lambda, \mu)= \{ \lambda \cap \mu \}$.
\end{nota}

\begin{exam}
\label{exam:crit}
The following examples given pairs $(\lambda , \mu) \in \crit (a,b)$ for various values of $a$ and $b$.  In the diagrams, $\mu/(\lambda \cap \mu)$ is represented in gray and $\lambda/ (\lambda\cap \nu)$ in light gray; $(\lambda \cap \mu)$ is in white.
Hence the homological degree of the element of $\M (\lambda , \mu)$ is  the number of dark gray boxes.

In the following, the reader should check for themselves that the outer corners of $\lambda \cap \mu$ cannot be removed, using the criterion of Lemma \ref{lem:rd}.
\begin{enumerate}
\item 
$a=1$, $b=2$, with $\lambda =(1,1)$, $\mu=(1)$:

\begin{ytableau}
\ \\
*(lightgray)
 \end{ytableau}.
\item
$a=4$, $b=3$, with $\lambda=(3)$, $\mu =(4)$:

\begin{ytableau}
\ &  \ &  \ & *(gray)
 \end{ytableau}.
\item 
$a=b=4$, with $\lambda=(2,1^2)$, $\mu= (3,1)$:

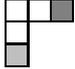
\begin{ytableau}
 \ &  \ & *(gray)\\
\ \\
*(lightgray)\\
\end{ytableau}.
\item 
$a=5$, $b=4$, with $\lambda=(2,1^2)$, $\mu= (3,2)$:

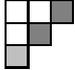
\begin{ytableau}
 \ &  \ & *(gray)\\
\ & *(gray)\\
*(lightgray)\\
\end{ytableau}.

\item 
$b=5$, $a=4$, with $\lambda= (3,2)$, $\mu=(4)$:

\begin{ytableau}
\ & \ &  \ & *(gray)\\
*(lightgray)& *(lightgray)
\end{ytableau}. 
\item 
$b=6$, $a=4$, with $\lambda= (3,3)$, $\mu=(4)$:

\begin{ytableau}
\ & \ &  \ & *(gray)\\
*(lightgray)& *(lightgray)& *(lightgray)
\end{ytableau}.
\item 
$a=b=6$, with $\lambda = (2,2,2)$, $\mu= (3,3)$:

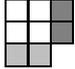
\begin{ytableau} 
 \ &\  & *(gray) \\
 \  & \ & *(gray) \\
 *(lightgray)& *(lightgray)
\end{ytableau}.
\end{enumerate}
\end{exam}

As a warm-up for describing $\crit (a,b)$ in general, consider the critical pairs for which $\mu \preceq \lambda$, which corresponds to homological degree zero. (This should be compared with the results of Section \ref{sect:perm}.)

\begin{lem}
\label{lem:crit_max_horizontal}
Suppose that $\mu \preceq \lambda$. Then $(\lambda, \mu ) \in \crit (a,b)$ if and only if $\mu = \hs{\lambda} $.
\end{lem}

\begin{proof}
Since $\mu \preceq \lambda$ by hypothesis, $\mu = (\lambda \cap \mu)$. 

Suppose that $\M(\lambda, \mu)$ is non-empty, so that $\mu = (\lambda \cap \mu) \in \M(\lambda, \mu)$.  Since $\lambda /\mu$ is a horizontal strip, one has $\hs{\lambda} \preceq \mu$. Suppose that this is not an equality, then there exists $\mu'$ obtained by removing an outer corner from $\mu$ such that $\hs{\lambda}\preceq \mu' \preceq \mu$. One checks that $\mu' \in \M (\lambda, \mu)$, hence $(\lambda, \mu) \not \in \crit (a,b)$, a contradiction. 

Conversely, if $\mu = \hs {\lambda}$, one checks easily that $(\lambda, \mu) \in \crit (a,b)$.
\end{proof}

\begin{prop}
\label{prop:crit}
The set $\crit (a,b)$ is in bijection with the set of pairs of partitions $(\gamma, \delta)$ that satisfy the following conditions:
\begin{enumerate}
\item 
$\hs {\gamma} \cap \vstrip{\delta} =\gamma \cap \delta  $; 
\item 
$|\gamma \cup \vstrip{\delta}|  = b$; 
\item 
$|\hs{\gamma} \cup \delta|  = a$.
\end{enumerate}

The bijection is defined by setting:
\begin{eqnarray*}
\lambda &:=& \gamma \cup \vstrip{\delta} \\
\mu &:=& \hs{\gamma} \cup \delta.
\end{eqnarray*}
In particular, $\mu \preceq \lambda$ if and only if $\delta =(0)$ and $\lambda \preceq \mu$ if and only if $\gamma =(0)$.
\end{prop}

\begin{proof}
One first checks that, given a pair $(\gamma, \delta)$, the associated $(\lambda, \mu)$ lies in $\crit(a,b)$. The case $\delta =(0)$ corresponds to Lemma \ref{lem:crit_max_horizontal} and the case $\gamma =(0)$ is proved by the `transpose' argument. 

Hence suppose that both $\gamma$ and $\delta$ are non-trivial. Consider the case where $\delta$ has one row and $\gamma$ one column. The condition $\gamma \cap \delta \preceq \hs {\gamma} \cap \vstrip{\delta} $ implies that $\gamma \cup \delta$ has the form of a hook:

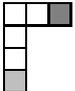
\begin{ytableau}
 \ &  \ & *(gray)\\
\ \\
\ \\
*(lightgray)\\
\end{ytableau}

\noindent
and this corresponds to a critical pair, by the criterion of Lemma  \ref{lem:rd}. (Note that the light gray box, corresponding to $\lambda/(\lambda \cap \mu)$ is to the left and below the dark gray box, which corresponds to $\mu / (\lambda \cap \mu)$.)

The general case follows by applying the argument of Lemma \ref{lem:crit_max_horizontal} (and its transpose) to the parts of the diagram arising from $\gamma$ and $\delta$ respectively. 

To show that this defines a bijection, given $\lambda$ and $\mu$ representing an element of $\crit(a, b)$, we require to exhibit the appropriate $\gamma $ and $\delta$. The key observation is that (again using the colouring of Example \ref{exam:crit}), all light gray boxes must lie to the left and below the dark gray boxes, as explained below. 

Suppose otherwise, then the edge of $\lambda \cup \mu$ would contain a sub-hook of the form

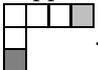
\begin{ytableau}
 \ & \ & \ & *(lightgray)\\
\ \\
*(gray)\\
\end{ytableau}.

\noindent
By the criterion of Lemma \ref{lem:rd}, one derives a contradiction to criticality (both white outer corners can be removed - and this applies even in the case \begin{ytableau}
 \ & *(lightgray)\\
*(gray)\\
\end{ytableau}). 

Hence define $\delta \leq \mu$ to be the smallest partition of the form $(\mu_1,  \ldots , \mu_s)$ with Young diagram containing all the dark gray boxes; $\gamma \preceq \lambda$ is defined likewise using the transpose construction (thus columns) for the light gray boxes. It remains to check that $(\gamma, \delta)$ satisfies the given conditions. This is left as an exercice for the reader. 
\end{proof}

There is a pleasing symmetry in the above construction, using the passage to the transpose partition $\lambda \mapsto \lambda^\dagger$. 

\begin{cor}
\label{cor:dagger}
There is a bijection $\crit(a,b) \cong  \crit (b,a)$ given by $(\lambda, \mu)  \mapsto (\mu^\dagger, \lambda^\dagger)$. 

Under the bijection of Proposition \ref{prop:crit}, this corresponds to 
$(\gamma, \delta) \mapsto (\delta^\dagger, \gamma ^\dagger)$. 
\end{cor}

%%%%%%%%%%%%%%%%%%%%%%%%%%%%%%%%%%%%%%%%%%%%%%%%%%%%%%%%%%%%%%%%%
\subsection{Homological consequences}

Recall that the Koszul complex $\kz _\bullet \kring \hfi (-, \mathbf{b}) \big)(\mathbf{a})$ is zero in homological degree $>a$ and each of the terms is a direct sum of finitely many simple $\sym_b \times \sym_a$-modules. This allows one to reason in terms of the Grothendieck group of such modules and to consider the Euler-Poincaré characteristic of the complex, which coincides with that of its homology. 

\begin{nota}
The class of a finite $\sym_b \times \sym_a $-module $X$ in the Grothendieck group is denoted by $[X]$.
\end{nota}

Using the notion of criticality, one deduces the following:

\begin{prop}
\label{prop:subset_H_finj}
For $a,b \in \nat$ and homological degree $n$, there is an inclusion of $\sym_b \times \sym_a $-modules:
\[
\bigoplus_{\substack{(\lambda, \mu) \in \crit (a,b) \\
|\mu/(\lambda \cap \mu)|=n }}
S^\lambda \boxtimes S^\mu 
\subset 
H_n ^\finj (\kring \hfi(-, \mathbf{b}) ) (\mathbf{a}).
\]
For $n=0$ this is an isomorphism.

Moreover, there is an equality in the Grothendieck group of $\sym_b \times \sym_a $-modules
\[
\sum_{(\lambda, \mu) \in \crit (a,b)}
(-1)^{|\mu/(\lambda \cap \mu)|}[S^\lambda \boxtimes S^\mu] 
= 
\sum_{n \in \nat} (-1)^n [H_n ^\finj (\kring \hfi(-, \mathbf{b}) ) (\mathbf{a})].
\]
\end{prop}

\begin{proof} By Lemma \ref{lem:M_kz}, for $(\lambda, \mu) \in \crit (a,b)$,  $S^\lambda \boxtimes S^\mu$ occurs with total multiplicity one in the Koszul complex calculating the $\finj$-homology $  H_*^\finj (\kring \hfi(-, \mathbf{b}) ) (\mathbf{a})$ and  this factor is in the given homological degree. This gives the inclusion. 

The case $n=0$ follows from Lemma \ref{lem:crit_max_horizontal} and Theorem \ref{thm:coker}.

The final statement follows by analysing Proposition \ref{prop:M_lambda_mu}, which gives
$$
\M (\lambda , \mu ) = \{ \nu'' | \nu (\lambda ,\mu) \preceq \nu'' \preceq (\lambda \cap \mu) \},
$$
where $\nu''$ represents an element in homological degree $|\mu / \nu''| = 
|\mu / (\lambda \cap \mu) | - |(\lambda \cap \mu) / \nu''|$. 

Now $(\lambda, \mu)$ is critical if and only if $|(\lambda \cap \mu)/ \nu(\lambda, \mu)|>0$. If $(\lambda, \mu)$ is not critical, then $2^{|(\lambda \cap \mu)/ \nu(\lambda, \mu)|}$ is even and one checks that the occurrences of $S^\lambda \boxtimes S^\mu$  in the Koszul complex sum to zero on forming the Euler-Poincaré characteristic. Since this is equal to the Euler-Poincaré characteristic of the homology, the result follows.
\end{proof}

On the basis of Proposition \ref{prop:subset_H_finj}, one can optimistically conjecture the following:

\begin{conj}
For $a,b \in \nat$ and homological degree $n$, there is an isomorphism of $\sym_b \times \sym_a $-modules:
\[
H_n ^\finj (\kring \hfi(-, \mathbf{b}) ) (\mathbf{a})
\cong
\bigoplus_{\substack{(\lambda, \mu) \in \crit (a,b) \\
|\mu/(\lambda \cap \mu)|=n }}
S^\lambda \boxtimes S^\mu. 
\]
\end{conj}

%%%%%%%%%%%%%%%%%%%%%%%%%%%%%%%%%%%%%%%%%%
\subsection{A possible approach}

We outline in this subsection a possible strategy for attacking the conjecture, based on the {\em dévissage} provided by Proposition \ref{prop:dévissage} below.

Write $\kring \langle 1 \rangle $ for the $\finj$-module that is zero on $\mathbf{n}$ unless $n=1$, when it takes value $\kring$ (with the only possible action of $\sym_1$).
 Then the following is standard:
 
 \begin{lem}
 For $F \in \ob \fimod$, $\kring \langle 1 \rangle \conv F $ carries a canonical $\finj$-module structure induced by that of $F$. 
 \end{lem} 

\begin{proof}
For a finite set $U$, $\kring \langle 1 \rangle \conv F (U) = \bigoplus_{\substack{U' \subset U \\|U'|= |U|-1}} F(U')$. Given an injection of finite sets $i : U \hookrightarrow V$, and $U ' = U \backslash \{ x \}$, take $V' := V\backslash \{ i(x) \}$, so that $i$ induces $i': U ' \hookrightarrow X'$. This construction gives rise to the required
\[
\kring \langle 1 \rangle \conv F (U)
\rightarrow 
\kring \langle 1 \rangle \conv F (V).
\]
\end{proof}

This allows the statement of the following result, which is the basis for an inductive analysis of the functors $\kring \hfi (-, \mathbf{b})$:

\begin{prop}
\label{prop:dévissage}
For $b \in \nat$, there is a short exact sequence of $\finj$-modules with values in $\kring [\sym_{b-1}]$-modules:
\[
0
\rightarrow
\kring \langle 1 \rangle \conv \kring \hfi (-, \mathbf{b-1}) 
\rightarrow 
\kring \hfi (-, \mathbf{b})\downarrow ^{\sym_b}_{\sym_{b-1}}
\stackrel{\pi}{\rightarrow} 
\kring \hfi (-, \mathbf{b-1}) 
\rightarrow 
0
\]
where the surjection $\pi$ is the retract of the canonical inclusion 
$\kring \hfi (-, \mathbf{b-1}) \subset \kring \hfi (-, \mathbf{b})$ that sends a generator $[f]$, where $\mathrm{image} (f) \not \subset \mathbf{b}$, to zero.
\end{prop}

\begin{proof}
The key point is that $\pi$ is a morphism of $\finj$-modules. This is an elementary, but important verification. That it is $\sym_{b-1}$-equivariant is clear.

To complete the proof, it remains to identify the kernel; this follows from the observation that, given $W \subset \mathbf{a}$ with $|W|= a-1$, there is an isomorphism
\[
\hfi (W, \mathbf{b-1}) 
\cong 
\hfi (\mathbf{a}, \mathbf{b}) 
\]
given by sending the element of $\mathbf{a} \backslash W$ to $b \in \mathbf{b}$.  Conversely, given $f \in \hfi (\mathbf{a}, \mathbf{b})$ such that $\mathrm{image} f \not \subset \mathbf{b-1}$, taking $W:= \mathbf{a} \backslash f^{-1} (b)$, $f$ is the image of $f|_W$.

Putting these points together, one obtains the result. 
\end{proof}

Applying $H^\finj_*$ to the short exact sequence of Proposition \ref{prop:dévissage} gives a long exact sequence:
\begin{eqnarray*}
\ldots 
\rightarrow
\kring \langle 1 \rangle \conv H^\finj_n \kring \hfi (-, \mathbf{b-1}) 
\rightarrow 
H^\finj_n \kring \hfi (-, \mathbf{b})\downarrow ^{\sym_b}_{\sym_{b-1}}
{\rightarrow} 
H^\finj_n \kring \hfi (-, \mathbf{b-1}) 
\rightarrow 
\\
\kring \langle 1 \rangle \conv  H^\finj_{n-1} \kring \hfi (-, \mathbf{b-1}) 
\rightarrow
\ldots
\end{eqnarray*}
using that $\kring \langle 1 \rangle \conv -$ commutes with the formation of $H^\finj_*$.

The above long exact sequence can be used to analyse $H^\finj_* \kring \hfi (-, \mathbf{b})$, by increasing induction on $b$. The difficulty is that, in order to prove the conjecture,  one requires to show that the connecting morphisms are all as non-trivial as is possible. 

It is possible to quantify exactly what is required (this is based on the explicit description of the $\crit (-,-)$ given by Proposition \ref{prop:crit}). However, this requires non-trivial input which should be expected to be as difficult as proving Theorem \ref{thm:coker}.

\begin{rem}
The above also provides an alternative strategy for proving Theorem \ref{thm:coker}, reducing to establishing the non-triviality of the first connecting morphism. Taking into account the addition structure introduced in \cite{P_wall} could facilitate this approach.
\end{rem}

\appendix
\section{The Koszul complex and Schur functors}
\label{sect:kz_schur}

The purpose of this Section is to identify the complex of Schur (bi)functors that is associated to the Koszul complex considered in Section \ref{sect:homology} in the specific cases of interest.

The passage to Schur functors is part of the  Schur-Weyl correspondence (see \cite[Chapter 4]{MR2522486} for example), which is a powerful tool when considering representations of the symmetric groups, especially when working over $\kring$  a field of characteristic zero (see \cite{2012arXiv1209.5122S}, for example).

%%%%%%%%%%%%%%%%%%%%%%%%%%%%%%%%%%%%%%%%%%%%%%%%
\subsection{From $\fbmod$ to Schur functors}

For $G \in \ob \fbmod$, the associated Schur functor (from $\fvs$ to $\vs$, where $\vs$ is the category of $\kring$-vector spaces and $\fvs$ the full subcategory of finite-dimensional spaces) is given by 
\[
G(V):= \bigoplus_{n\in \nat} V^{\otimes n} \otimes_{\sym_n} G(\mathbf{n}),
\]
where $\sym_n$ acts by place permutations of tensor factors on $V^{\otimes n}$. 

\begin{exam}
\label{exam:schur_simple}
If $\lambda \vdash n$ and $S^\lambda$ is the associated simple representation of $\sym_n$, since $\kring$ is a field of characteristic zero, $S^\lambda (V)$ is the usual Schur functor associated to the partition $\lambda$. It is a  simple functor. 
\end{exam}

As is well-known (see \cite{2012arXiv1209.5122S}, for example), 
 this construction is symmetric monoidal with respect to the convolution product on $\fbmod$ and the pointwise tensor product of functors:

\begin{prop}
\label{prop:schur_sym_mon}
For $G_1 , G_2 \in \ob \fbmod$, there is a natural isomorphism of functors:
\[
(G_1 \conv G_2) (V) 
\cong 
G_1 (V) \otimes G_2 (V).
\]
\end{prop}

The Schur functor construction clearly passes to the bivariant case, namely functors from $\fb\op \times \fb$ (or, equivalently, $\fb \times \fb$) to $\vs$. For example, for a left $\sym_a\op \times \sym_b$-module, $	M$,  the Schur bifunctor is 
\[
(V, W) \mapsto W^{\otimes b}\otimes _{\sym_b} M \otimes _{\sym_a} V^{\otimes a}
\]
for $(V, W) \in (\fvs) ^{\times 2}$.  

\begin{exam}
\label{exam:sym_a_bischur}
Take $b=a$ and consider $\kring \sym_a$ as a $\sym_a$-bimodule with respect to the regular structures. Then the associated Schur bifunctor is 
\[
W^{\otimes a} \otimes _{\sym_a} \kring \sym_a \otimes_{\sym_a} V^{\otimes a}.
\]
This is isomorphic to $W^{\otimes a} \otimes _{\sym_a}  V^{\otimes a}$ and hence to the bifunctor $S^a(W \otimes V)$, where $S^a (-)$ denotes the $a$th symmetric product functor $S^a (Z):=  Z^{\otimes a}/\sym_a$ for $Z \in \ob \fvs$.
\end{exam}

This example extends to give:

\begin{prop}
\label{prop:Schur_bif_ab}
For $a\leq b \in \nat$, the Schur bifunctor associated to the $\sym_a\op \times \sym_b$-module $\kring \hfi (\mathbf{a},\mathbf{b}) $ is 
naturally isomorphic to
\[
S^{b-a} (W) \otimes S^a (W \otimes V).
\] 
\end{prop}

\begin{proof}
By Lemma \ref{lem:perm}, the $\sym_a\op \times \sym_b$-module $\kring \hfi (\mathbf{a},\mathbf{b}) $ is isomorphic to the permutation bimodule on $\sym_b/ \sym_{b-a}$. 
The result follows by a straightforward extension of the argument outlined in Example \ref{exam:sym_a_bischur}. 
\end{proof}

%%%%%%%%%%%%%%%%%%%%%%%%%%%%%%%%%%%%%%%%%%%%%%%%%%%%%%%%%%%%
\subsection{Revisiting the Koszul complex}
\label{subsect:revisit_kz}

The Schur bifunctor construction applies to the Koszul complex $\kz_\bullet \kring \hfi (-,-)$ since, for fixed $a, b \in \nat$, this gives a complex of $\sym_a\op \times \sym_b$-modules and hence a natural complex in bifunctors  $(V,W)\mapsto \kz_\bullet \kring \hfi (-,-)(V,W)$, considering all $a$, $b$ at once. 

\begin{rem}
The natural numbers $a$ and $b$ can be recovered respectively as the polynomial degree with respect to $V$ and the polynomial degree with respect to $b$ (compare Proposition \ref{prop:Schur_bif_ab}.
\end{rem}

\begin{thm}
\label{thm:schur_koszul}
There is an isomorphism of complexes
\[
(\kz_\bullet \kring \hfi (-,-)(V,W) ,d)
\cong 
(S^* (W) \otimes \Lambda^* (V) \otimes S^* (W \otimes V), d)
\]
with Koszul-type differential from homological degree $n+1$ to $n$ given in `polynomial bidegree' $(a,b)$ as the composite
\begin{eqnarray*}
S^{b-a +n+1}  (W) \otimes \Lambda^{n+1} (V) \otimes S^{a-n-1} (W \otimes V)
 &\rightarrow &
 S^{b-a +n}  (W) \otimes \Lambda^{n} (V) \otimes (W \otimes V)\otimes  S^{a-n-1} (W \otimes V)
\\
 &\rightarrow & 
S^{b-a +n}  (W) \otimes \Lambda^{n} (V) \otimes  S^{a-n} (W \otimes V),
\end{eqnarray*}
where the first map is induced by the coproducts 
\begin{eqnarray*}
S^*(W) &\rightarrow& S^{*-1}(W) \otimes W 
\\
\Lambda^* (V) &\rightarrow &  \Lambda^{*-1}(V) \otimes V
\end{eqnarray*}
 and the second by the product of the symmetric algebra $S^*(W \otimes V)$.

In particular, this is a complex in the category of $S^*(W \otimes V)$-modules.
\end{thm}

\begin{proof}
The explicit identification of the terms in $\kz_\bullet \kring \hfi (-,-)(V,W)$ follows 
from the definition of $\kz_\bullet \kring \hfi (-,-)$, the identification of the Schur functor associated with the orientation module $\ori$ as the functor $V \mapsto \Lambda^* (V)$, the exterior algebra on $V$, together with Proposition \ref{prop:schur_sym_mon} to treat the convolution product. 

The Koszul-type differential is simply a translation of the explicit differential in $(\kz_\bullet \kring \hfi (-,-),d)$ in terms of the Schur bifunctors. 
\end{proof}

\begin{rem}
Working over a field of characteristic zero, the Koszul complex $(\kz_\bullet \kring \hfi (-,-),d)$ is determined by the description of the Schur bifunctor given in Theorem \ref{thm:schur_koszul}.
\end{rem}

Theorem \ref{thm:schur_koszul} contains further important information on the $\finj$-homology of $\kring \hfi (-,-)\trans$ that has been omitted in the body of the text:

\begin{cor}
\label{cor:module}
The homology of $(\kz_\bullet \kring \hfi (-,-)(V,W) ,d)$ takes values in the category of $S^* (W \otimes V)$-modules. 
\end{cor}

\begin{rem}
One of the contributions of  \cite{P_wall} is to present a natural categorical framework for this structure that does not require passage to Schur bifunctors. 
\end{rem}

%%%%%%%%%%%%%%%%%%%%%%%%%%%%%%%%%%%%%%%%%%%%%%%%%%%%%%%%%%%%%%%%%%%%%
\subsection{Flipping the complex}

Corollary \ref{cor:dagger} highlighted a duality of $\crit (-,-)$ under $^\dagger$. 
A form of this already holds at the level of the Koszul complexes. 

To see this it is useful to rewrite the Koszul complex in the following way:
\[
(S^* (W) \otimes S^* (sV) \otimes S^* (W \otimes V), d),
\]
where $sV$ denotes $V$ concentrated in homological degree one and $S^* (-)$ is defined using Koszul signs in the category of graded vector spaces. 

Now, consider the following substitution:
\begin{eqnarray*}
W&:=& s X \\
V&:= & s^{-1} Y,
\end{eqnarray*}
where $X$, $Y$ are $\kring$-vector spaces, so that $W$ is in homological degree one and $V$ in homological degree $-1$. 

This yields the complex:
\[
(S^* (sX) \otimes S^* (Y) \otimes S^* (X \otimes Y), d),
\]
which identifies with the original Koszul complex, with the rôles of $X$, $Y$ reversed.

%\nocite{*}
%\bibliographystyle{amsalpha}
%\bibliography{baby.bib}

\begin{thebibliography}{{Pow}22}

\bibitem[CE17]{MR3654111}
Thomas Church and Jordan~S. Ellenberg, \emph{Homology of {FI}-modules}, Geom.
  Topol. \textbf{21} (2017), no.~4, 2373--2418. \MR{3654111}

\bibitem[CSST10]{MR2643487}
Tullio Ceccherini-Silberstein, Fabio Scarabotti, and Filippo Tolli,
  \emph{Representation theory of the symmetric groups}, Cambridge Studies in
  Advanced Mathematics, vol. 121, Cambridge University Press, Cambridge, 2010,
  The Okounkov-Vershik approach, character formulas, and partition algebras.
  \MR{2643487}

\bibitem[Gan16]{MR3603074}
Wee~Liang Gan, \emph{A long exact sequence for homology of {FI}-modules}, New
  York J. Math. \textbf{22} (2016), 1487--1502. \MR{3603074}

\bibitem[GW09]{MR2522486}
Roe Goodman and Nolan~R. Wallach, \emph{Symmetry, representations, and
  invariants}, Graduate Texts in Mathematics, vol. 255, Springer, Dordrecht,
  2009. \MR{2522486}

\bibitem[{Pow}22]{P_wall}
Geoffrey {Powell}, \emph{{Baby bead representations}}, preprint (2022).


\bibitem[{Sage}22]{sagemath}
{The Sage Developers}, \emph{{S}agemath, the {S}age {M}athematics {S}oftware
  {S}ystem ({V}ersion 9.4)}, 2022, {\tt https://www.sagemath.org}.


\bibitem[SS12]{2012arXiv1209.5122S}
Steven~V {Sam} and Andrew {Snowden}, \emph{{Introduction to twisted commutative
  algebras}}, arXiv e-prints (2012), arXiv:1209.5122.

\end{thebibliography}

\providecommand{\bysame}{\leavevmode\hbox to3em{\hrulefill}\thinspace}
\providecommand{\MR}{\relax\ifhmode\unskip\space\fi MR }
\providecommand{\MRhref}[2]{%
  \href{http://www.ams.org/mathscinet-getitem?mr=#1}{#2}
}
\providecommand{\href}[2]{#2}

\end{document}